\documentclass[11pt]{article}

\usepackage{amssymb,amsmath,amsfonts,amsthm}
\usepackage{latexsym}
\usepackage{graphics}
\usepackage{indentfirst}
\usepackage{hyperref}
\usepackage{comment}
\usepackage[shortlabels]{enumitem}
\usepackage[english]{babel}
\usepackage{tikz}
\usepackage{esdiff}

\usetikzlibrary{arrows.meta}

\newtheorem*{thm*}{Theorem}
\newtheorem{thm}{Theorem}
\newtheorem{lem}[thm]{Lemma}
\newtheorem{pro}[thm]{Proposition}

\newtheorem{cor}[thm]{Corollary}
\newtheorem{conj}[thm]{Conjecture}

\newtheorem{ques}[thm]{Question}

\newcommand{\N}{\mathbb{N}}

\allowdisplaybreaks

\begin{document}

\title{Bounding the List Color Function Threshold from Above}

\author{Hemanshu Kaul$^1$, Akash Kumar$^2$, Andrew Liu$^2$, Jeffrey A. Mudrock$^2$, Patrick Rewers$^2$, \\ Paul Shin$^2$, Michael Scott Tanahara$^2$, and Khue To$^2$}

\footnotetext[1]{Department of Applied Mathematics, Illinois Institute of Technology, Chicago, IL 60616. E-mail: {\tt {kaul@iit.edu}}}

\footnotetext[2]{Department of Mathematics, College of Lake County, Grayslake, IL 60030.  E-mail:  {\tt {jmudrock@clcillinois.edu}}}

\maketitle

\begin{abstract}

The chromatic polynomial of a graph $G$, denoted $P(G,m)$, is equal to the number of proper $m$-colorings of $G$ for each $m \in \N$.  In 1990, Kostochka and Sidorenko introduced the list color function of graph $G$, denoted $P_{\ell}(G,m)$, which is a list analogue of the chromatic polynomial. The \emph{list color function threshold of $G$}, denoted $\tau(G)$, is the smallest $k$ such that $P(G,k) > 0$ and $P_{\ell}(G,m) = P(G,m)$ whenever $m \geq k$.  It is known that for every graph $G$, $\tau(G)$ is finite, and a recent paper of Kaul et al. suggests that complete bipartite graphs may be the key to understanding the extremal behavior of $\tau$. In this paper we develop tools for bounding the list color function threshold of complete bipartite graphs from above.  We show that for any $n \geq 2$, $\tau(K_{2,n}) \leq \lceil (n+2.05)/1.24 \rceil$.  Interestingly, our proof makes use of classical results such as Rolle's Theorem and Descartes' Rule of Signs.

\medskip



\end{abstract}

\section{Introduction}\label{intro}

In this paper all graphs are nonempty, finite, simple graphs.  Generally speaking, we follow West~\cite{W01} for terminology and notation.  We write \emph{AM-GM inequality} for the inequality of arithmetic and geometric means.  The set of natural numbers is $\N = \{1,2,3, \ldots \}$.  For $m \in \N$, we write $[m]$ for the set $\{1, \ldots, m \}$, and when $0 \leq k \leq m$, we write $\binom{[m]}{k}$ for the set of all $k$-element subsets of $[m]$.  We write $K_{l,n}$ for complete bipartite graphs with partite sets of size $l$ and $n$.  If $G$ and $H$ are vertex disjoint graphs, we write $G \vee H$ for the join of $G$ and $H$.  

\subsection{List Coloring and The List Color Function} \label{basic}

In the classical vertex coloring problem, we wish to color the vertices of a graph $G$ with up to $m$ colors from $[m]$ so that adjacent vertices receive different colors, a so-called \emph{proper $m$-coloring}.  The \emph{chromatic number} of a graph, denoted $\chi(G)$, is the smallest $m$ such that $G$ has a proper $m$-coloring.  List coloring is a well-known variation on classical vertex coloring which was introduced independently by Vizing~\cite{V76} and Erd\H{o}s, Rubin, and Taylor~\cite{ET79} in the 1970s.  For list coloring, we associate a \emph{list assignment} $L$ with a graph $G$ such that each vertex $v \in V(G)$ is assigned a list of available colors $L(v)$ (we say $L$ is a list assignment for $G$).  We say $G$ is \emph{$L$-colorable} if there is a proper coloring $f$ of $G$ such that $f(v) \in L(v)$ for each $v \in V(G)$ (we refer to $f$ as a \emph{proper $L$-coloring} of $G$).  A list assignment $L$ for $G$ is called a \emph{$k$-assignment} if $|L(v)|=k$ for each $v \in V(G)$.  The \emph{list chromatic number} of a graph $G$, denoted $\chi_\ell(G)$, is the smallest $k$ such that $G$ is $L$-colorable whenever $L$ is a $k$-assignment for $G$.  It is immediately obvious that for any graph $G$, $\chi(G) \leq \chi_\ell(G)$.  Moreover, it is well-known that the gap between the chromatic number and list chromatic number of a graph can be arbitrarily large since $\chi_{\ell}(K_{n,t}) = n+1$ whenever $t \geq n^n$ (see e.g.,~\cite{KM216} for further details).  

In 1912, Birkhoff~\cite{B12} introduced the notion of the chromatic polynomial in order to make progress on the four color problem.  For $m \in \N$, the \emph{chromatic polynomial} of a graph $G$, $P(G,m)$, is the number of proper $m$-colorings of $G$. It is easy to show that $P(G,m)$ is a polynomial in $m$ of degree $|V(G)|$ (e.g., see~\cite{DKT05}). For example, $P(K_n,m) = \prod_{i=0}^{n-1} (m-i)$, $P(C_n,m) = (m-1)^n + (-1)^n (m-1)$, $P(T,m) = m(m-1)^{n-1}$ whenever $T$ is a tree on $n$ vertices, and $P(K_{2,n},m) = m(m-1)^n + m(m-1)(m-2)^n$ (see~\cite{B94} and~\cite{W01}). 

The notion of chromatic polynomial was extended to list coloring in the early 1990s by Kostochka and Sidorenko~\cite{AS90}.  If $L$ is a list assignment for $G$, let $P(G,L)$ denote the number of proper $L$-colorings of $G$. The \emph{list color function} $P_\ell(G,m)$ is the minimum value of $P(G,L)$ where the minimum is taken over all possible $m$-assignments $L$ for $G$.  Since an $m$-assignment could assign the same $m$ colors to every vertex in a graph, it is clear that $P_\ell(G,m) \leq P(G,m)$ for each $m \in \N$.  In general, the list color function can differ significantly from the chromatic polynomial for small values of $m$.  For example, for any $n \in \N$, $P_{\ell}(K_{n,n^n},m) = 0$ whenever $m \in [n]$.  On the other hand,  in 1992, Donner~\cite{D92} showed that for any graph $G$ there is a $k \in \N$ such that $P_\ell(G,m) = P(G,m)$ whenever $m \geq k$. 

\subsection{The List Color Function Threshold} 

With Donner's 1992 result in mind, it is natural to study the point at which the list color function of a graph becomes identical to its chromatic polynomial.  Given any graph $G$, the \emph{list color function number of $G$}, denoted $\nu(G)$, is the smallest $t \geq \chi(G)$ such that $P_{\ell}(G,t) = P(G,t)$.  The \emph{list color function threshold of $G$}, denoted $\tau(G)$, is the smallest $k \geq \chi(G)$ such that $P_{\ell}(G,m) = P(G,m)$ whenever $m \geq k$.  Clearly, $\chi(G) \leq \chi_\ell(G) \leq \nu(G) \leq \tau(G)$.

One of the most famous and important open questions on the list color function asks whether the list color function number of a graph can differ from its list color function threshold. 

\begin{ques} [Kirov and Naimi~\cite{KN16}] \label{ques: threshold}
	For every graph $G$, is it the case that $\nu(G) = \tau(G)$?
\end{ques} 

Much of the research on the list color function has been focused on studying the list color function threshold.  One specific topic of interest is \emph{enumeratively chromatic-choosable} graphs.  A graph $G$ is called \emph{chromatic-choosable} if $\chi_{\ell}(G) = \chi(G)$, and $G$ is said to be \emph{enumeratively chromatic-choosable} if $\tau(G) = \chi(G)$.  The earliest recorded result on the list color function states that chordal graphs are enumeratively chromatic-choosable~\cite{AS90}.  Similarly, cycles are enumeratively chromatic-choosable~\cite{KN16}.  It is also known that for any graph $G$, if $\tau(G) = m$, then $\tau(G \vee K_n) \leq m+n$ (see Proposition~18 in~\cite{KM18}).  This result implies that if $G$ is enumeratively chromatic-choosable, then $G \vee K_n$ is enumeratively chromatic-choosable for any $n \in \N$.  It is challenging to determine which graphs are enumeratively chromatic-choosable and few such graphs are known.

\begin{ques} \label{ques: ecc}
Which graphs are enumeratively chromatic-choosable?
\end{ques}  

Researchers have also studied general upper bounds on the list color function threshold.  In 2009, Thomassen~\cite{T09} showed that for any graph $G$, $\tau(G) \leq |V(G)|^{10} + 1$.  Then, in 2017, Wang, Qian, and Yan~\cite{WQ17} showed that for any graph $G$, $\tau(G) \leq (|E(G)|-1)/\ln(1+ \sqrt{2}) + 1$.  Thomassen also asked the following question in his 2009 paper.

\begin{ques} [Thomassen~\cite{T09}] \label{ques: universal}
	Is there a universal constant $\alpha$ such that for any graph $G$, $\tau(G) - \chi_{\ell}(G) \le \alpha$?
\end{ques}

Question~\ref{ques: universal} went unanswered for 13 years until a very recent breakthrough.  It was shown in~\cite{KK22} that there is a positive constant $C$ such that $\tau(K_{2,n}) - \chi_{\ell}(K_{2,n}) \geq C \sqrt{n}$ whenever $n \geq 16$ (see Theorem~\ref{theorem:firstthresholdpaper} below).  This shows that the answer to Question~\ref{ques: universal} is no in a rather strong sense, and it also shows the significance of studying the list color function of complete bipartite graphs.  The motivation for this paper was to develop techniques for bounding the list color function threshold of complete bipartite graphs from above and to make progress on the following related Conjecture.

\begin{conj} [\cite{KK22}] \label{conj: k2l}
	$\tau(K_{2,n}) = \Theta(\sqrt{n})$ as $n \rightarrow \infty$.
\end{conj} 

\subsection{Outline of Results}

We now present an outline of the paper.  We begin Section~\ref{one} by presenting some tools that are useful for studying upper bounds on the list color function threshold of complete bipartite graphs.  We begin by proving an important lemma that tells us something about the structure of a list assignment for complete bipartite graphs that has as few proper list colorings as possible.

\begin{lem} \label{lemma:unionbadgeneral} Suppose $l\leq n$, and let $G = K_{l, n}$ with bipartition $\{x_1,\hdots, x_l\}, \{y_1, \hdots, y_n\}$. For any $m \geq \chi_{\ell}(G)$, there exists an $m$-assignment $L$ for $G$, such that $L(y_j) \subseteq \bigcup_{i=1}^l L(x_i)$ for all $j \in [n]$ and $P(G, L) = P_{\ell}(G, m)$.
\end{lem}

With the intent of bounding $\tau(K_{2,n})$ from above, we prove a useful lemma which further characterizes the structure of a list assignment for $K_{2,n}$ that has as few proper list colorings as possible.  For the lemma below, note that it is easy to show that $P(K_{2,n},m) = m(m-1)^n + m(m-1)(m-2)^n$ for each $m \in \N$.

\begin{lem}\label{lemma:dm-1bad}
Let $G = K_{2,n}$ with bipartition $\{x_1, x_2\}, \{y_1, \hdots, y_n\}$ and $n \geq 3$. If $L$ is an $m$-assignment for $G$ with $m\geq 3$, $\vert L(x_1)\cap L(x_2)\vert = m-1$, and $L(y_j)\subseteq L(x_1)\cup L(x_2)$ for all $j\in [n]$, then $P(G,L)\geq P(G,m) = m(m-1)^n + m(m-1)(m-2)^n$.
\end{lem}

We end Section~\ref{one} by studying $P_{\ell}(K_{2,n}, m)$ when $m \in \{2,3,4,5\}$.  Specifically, through some delicate case work we are able to prove the following.

\begin{thm} \label{thm: casework}
The following statements hold: \\
(i) $P_{\ell}(K_{2,n},3) = P(K_{2,n}, 3)$ whenever $2 \leq n \leq 10$ and $P_{\ell}(K_{2,n}, 3) < P(K_{2,n}, 3)$ whenever $n \geq 12$; \\
(ii) $P_{\ell}(K_{2,n},4) = P(K_{2,n}, 4)$ whenever $2 \leq n \leq 24$ and $P_{\ell}(K_{2,n}, 4) < P(K_{2,n}, 4)$ whenever $n \geq 27$; \\
(iii) $P_{\ell}(K_{2,n},5) = P(K_{2,n}, 5)$ whenever $2 \leq n \leq 43$ and $P_{\ell}(K_{2,n},5) < P(K_{2,n}, 5)$ whenever $n \geq 44$.\\
\end{thm}
As indicated by Theorem~\ref{thm: casework}, it is unknown whether $P_{\ell}(K_{2,11},3) = P(K_{2,11}, 3)$, and it is unknown whether $P_{\ell}(K_{2,n},4) = P(K_{2,n}, 4)$ when $n \in \{25, 26\}$.  Theorem~\ref{thm: casework} also leads us to make the following conjecture.

\begin{conj} \label{conj: mono}
If $n,m \in \N$ and $P_{\ell}(K_{2,n+1},m) = P(K_{2,n+1},m)$, then $P_{\ell}(K_{2,n},m) = P(K_{2,n},m)$.
\end{conj}

It is easy to show that Conjecture~\ref{conj: mono} is true when $m \in [2]$.  We also know by Theorem~\ref{thm: casework} that Conjecture~\ref{conj: mono} is true when $m \in \{3,5\}$.

In Section~\ref{two} we seek to prove an upper bound on $\tau(K_{2,n})$.  In our proof of Theorem~\ref{thm: casework}, the most delicate parts of the argument rely on several careful applications of the AM-GM inequality.  Our idea for proving a general upper bound on $\tau(K_{2,n})$ is to see what happens if we insist on only applying the AM-GM inequality once.  After some careful analysis which includes the use of classical results such as Rolle's Theorem and Descartes' Rule of Signs, we are able to deduce the following.

\begin{thm}\label{thm:upperboundlinear}
Let $n\in \mathbb{N}$ with $n\geq 2$, and let $G = K_{2,n}$. Then, 
\[\tau(G) \leq \left\lceil \frac{n + 2.05}{1.24}\right\rceil.\]
\end{thm}

One will notice that the bound in Theorem~\ref{thm:upperboundlinear} is about 2.8 times better than the general upper bound of Wang, Qian, and Yan~\cite{WQ17} (i.e., $\tau(G) \leq (|E(G)|-1)/\ln(1+ \sqrt{2}) + 1$ for any graph $G$).  On the other hand, Theorem~\ref{thm:upperboundlinear} falls short of verifying Conjecture~\ref{conj: k2l}.  While we believe Conjecture~\ref{conj: k2l} is true, it seems that a proof of Conjecture~\ref{conj: k2l} will require a novel refinement of our approach or an entirely new approach.

We conclude Section~\ref{two} by showing how the results of this paper allow us to prove that $K_{2,3}$ is enumeratively chromatic-choosable and $\tau(K_{2,4}) = \tau(K_{2,5}) = 3$.   

\section{Tools and Small List Sizes} \label{one}

We begin by proving Lemma~\ref{lemma:unionbadgeneral}.

\begin{proof}
Let $L$ be an $m$-assignment for $G$ with $m \geq \chi_{\ell}(G)$ such that $P(G, L) = P_{\ell}(G, m)$. Let $Q = \bigcup_{i=1}^l L(x_i)$. If $L(y_j) \subseteq Q$ for each $j \in [n]$, then we are done. So, suppose there exists some $k \in [n]$ and $c \in L(y_k)$ such that $c \not \in Q$. Let $X = Q - L(y_k)$. Since $\lvert Q \rvert \geq m$, $\lvert L(y_k) \rvert = m$, and $\lvert L(y_k) - Q \rvert \geq 1$, we must have $\lvert X \rvert \geq 1$. Let $d$ be any element of $X$, and let $L'$ be the $m$-assignment for $G$ defined by $L'(v) = L(v)$ for each $v \in V(G) - \{y_k\}$ and $L'(y_k) = (L(y_k) - \{c\}) \cup \{d\}$. Let $S$ be the set of all proper $L$-colorings of $G$, and let $S'$ be the set of all proper $L'$-colorings of $G$. Note that since $m\geq \chi_{\ell}(G)$ , $S'$ is a non-empty set.

Let $h : S' \rightarrow S$ be the function defined by\begin{align*}    h(f) =    \begin{cases}    f & \text{if } f(y_k) \neq d, \\    g & \text{if } f(y_k) = d,    \end{cases}\end{align*} where $g$ is the proper $L$-coloring of $G$ defined by $g(v) = f(v)$ if $v \in V(G)- \{y_k\}$ and $g(y_k) = c$. We will now show that $h$ is injective.
Suppose $f_1, f_2 \in S'$ such that $f_1 \neq f_2$. Then, we must show that $h(f_1) \neq h(f_2)$. There are three cases to consider.
Case $(1)$: $f_1(y_k)\neq d $ and $f_2(y_k)\neq d$. Then $h(f_1) = f_1\neq f_2=h(f_2)$.
Case $(2)$: $f_1(y_k)=f_2(y_k)=d$. There exists $z \in V(G) - \{y_k\}$ such that $f_1(z)\neq f_2(z)$. Let $g_1 = h(f_1)$ and $g_2 = h(f_2)$. Then $g_1(z) = f_1(z)\neq f_2(z) = g_2(z)$. Thus, $h(f_1) \neq h(f_2)$.
Case $(3)$: $f_1(y_k)=d$ and $f_2(y_k) \neq d$. Notice $f_2 = h(f_2)$. Let $g_1 = h(f_1)$. Since $f_2(y_k) \in L'(y_k)$ and $c\notin L'(y_k)$, then $f_2(y_k) \neq c = g_1(y_k)$. Therefore, $h(f_1) = g_1\neq f_2 = h(f_2)$.
Thus, $h$ is injective. Therefore, $P(G, L') = \lvert S' \rvert \leq \lvert S \rvert = P(G, L) = P_{\ell}(G, m)$ which means that $P(G, L') = P_{\ell}(G, m)$. Hence, by repeatedly selecting such a pair $y_k,c$ and modifying the list assignment so that $c\not\in L'(y_k)$, we will eventually get an $m$-assignment $L^*$ for which $L^*(y_j)$ contains only colors of $Q$ for all $j \in [n]$ and $P(G, L^*) = P_{\ell}(G, m)$.
\end{proof}

Notice that since $\chi_{\ell}(K_{l,n}) \leq l+1$, Lemma~\ref{lemma:unionbadgeneral} can always be applied when $m \geq l+1$.  We now introduce notation that will be useful for us throughout the paper. Fix $G = K_{2,n}$ with bipartition $\{x_1, x_2\}, \{y_1, \hdots, y_n\}$ and let $L$ be an arbitrary $m$-assignment for $G$. Then, for each $(i, j) \in L(x_1) \times L(x_2)$, let $\mathcal{C}_{(i, j)}$ be the set of proper $L$-colorings of $G$ in which $x_1$ is colored with $i$ and $x_2$ is colored with $j$. Notice $P(G, L) = \sum_{(i, j) \in L(x_1) \times L(x_2)} \vert \mathcal{C}_{(i, j)}\vert$. 

We are now ready to prove Lemma~\ref{lemma:dm-1bad}.

\begin{proof}
For each $(i, j) \in L(x_1) \times L(x_2)$, $|\mathcal{C}_{(i, j)}| \geq (m - 2)^n$. Suppose without loss of generality that $L(x_1) = [m]$ and $L(x_2) = [m-1] \cup \{m+1\}$. Then $L(y_k) \in \binom{[m+1]}{m}$ for each $k \in [n]$. For each $a \in [m+1]$, let $z_a = \lvert \{k \in [n] : L(y_k) = [m + 1] - \{a\}\} \rvert$. Notice that $\sum_{a=1}^{m+1} z_a = n$.

Note
\begin{align*}
    \sum_{i=1}^{m-1} |\mathcal{C}_{(i, i)}| =  \sum_{i=1}^{m-1} m^{z_i} (m-1)^{n - z_i} &= (m-1)^{n} \sum_{i=1}^{m-1} \left(\frac{m}{m-1}\right)^{z_i}\\
    &= (m-1)^{n} \sum_{i=1}^{m-1} \left(1 + \frac{1}{m-1}\right)^{z_i}.
\end{align*}

Clearly, $(1+1/(m-1))^{z_i} \geq 1 + z_i/(m-1)$. Consequently, if we let $x = \sum_{a=1}^{m-1} z_a$, then

\[\sum_{i=1}^{m-1} |\mathcal{C}_{(i, i)}| \geq (m-1)^{n} \sum_{i=1}^{m-1} \left(1 + \frac{z_i}{m-1}\right) = (m-1)^{n} \left(m-1 + \frac{x}{m-1}\right).\]

If $x \geq m - 1$, then
\begin{align*}
    P(G, L)
    &= \sum_{(i, j) \in L(x_1) \times L(x_2)} |\mathcal{C}_{(i, j)}| \\
    &\geq (m-1)^{n} \left(m-1 + \frac{x}{m-1}\right)
        + \sum_{(i, j) \in (L(x_1) \times L(x_2)),\; i \neq j} |\mathcal{C}_{(i, j)}| \\
    &\geq m(m-1)^n + \sum_{(i, j) \in (L(x_1) \times L(x_2)),\; i \neq j} (m-2)^n \\
    &= m(m-1)^n + (m^2-m+1)(m-2)^n \\
    &> P(G, m).
\end{align*}

So we may assume $0 \leq x \leq m-2$. Notice
\begin{align*}
\vert\mathcal{C}_{(m, m+1)}\vert + \sum_{i=1}^{m-1} \vert\mathcal{C}_{(i, i)}\vert
    &\geq (m-2)^x (m-1)^{n-x} + (m-1)^{n} \left(m-1 + \frac{x}{m-1}\right). \\
    &= (m-1)^n \left[\left(\frac{m-2}{m-1}\right)^x + m-1+\frac{x}{m-1}\right]. 
\end{align*}
By Bernoulli's Inequality we have
\begin{align*}
    \left( \frac{m - 2}{m - 1} \right)^x = \left( 1 - \frac{1}{m - 1} \right)^x \geq 1 - \frac{x}{m - 1}
\end{align*}
which implies that
\begin{align*}
    \vert\mathcal{C}_{(m, m + 1)}\vert + \sum_{i = 1}^{m - 1} \vert\mathcal{C}_{(i, i)}\vert
    &\geq (m - 1)^n \left[\left(\frac{m-2}{m-1}\right)^x + m-1+\frac{x}{m-1}\right] \\
    &\geq (m - 1)^n \left[ 1 - \frac{x}{m - 1} + m - 1 + \frac{x}{m - 1} \right] \\
    &= m(m - 1)^n.
\end{align*}
Hence,
\begin{align*}
    P(G, L)
    &= \sum_{(i, j) \in L(x_1) \times L(x_2)} \vert \mathcal{C}_{(i, j)} \vert \\
    &= \left( \sum_{i = 1}^{m - 1} \vert \mathcal{C}_{(i, i)} \vert + \vert \mathcal{C}_{(m, m + 1)} \vert \right) + \sum_{\substack{(i, j) \in (L(x_1) \times L(x_2)) \\ (i, j) \neq (m, m + 1), \;i \neq j}} \vert \mathcal{C}_{(i, j)} \vert \\
    &\geq m(m - 1)^n + \sum_{\substack{(i, j) \in (L(x_1) \times L(x_2)) \\ (i, j) \neq (m, m + 1), \;i \neq j}} (m - 2)^n \\
    &= m(m - 1)^n + (m^2 - m)(m - 2)^n \\
    &= P(G, m).
\end{align*}
\end{proof}

Before proving Theorem~\ref{thm: casework}, we need several technical lemmas.  The first of which is a key Lemma in~\cite{KK22}.

\begin{lem}[\cite{KK22}]\label{lemma:paulconstruction}
Let $n, m, t \in \mathbb{N}$ with $n \geq 2$ and $m \geq n + 1$, and let $G = K_{n, n^nt}$ with bipartition $X, Y$ where $X = \{x_1, \ldots, x_n\}$ and $Y = \{y_1, \ldots, y_{n^nt}\}$. Let $S_k = \{m + n(k - 2) + \ell : \ell \in [n]\}$ for each $k \in [n]$, and let $A = \{\{s_1, \ldots, s_n\} : s_k \in S_k \text{ for each } k \in [n]\}$. Suppose $A = \{A_0, \ldots, A_{n^n - 1}\}$. Let $L$ be the $m$-assignment for $G$ defined by $L(x_k) = [m - n] \cup S_k$ for each $k \in [n]$ and $L(y_k) = [m - n] \cup A_{\lfloor (k - 1) / t \rfloor}$ for each $k \in [n^n t]$. Then
\begin{align*}
    P(G, L)
    &= n^n \prod_{i = 0}^n (m - i)^{t\binom{n}{i}(n - 1)^{n - i}} \\
    & + \sum_{N = 1}^n \sum_{S = 0}^{n - N} \left[ n^S \binom{n}{S} \binom{m - n}{N} \left( \sum_{i = 0}^{N - 1} (-1)^i \binom{N}{i} (N - i)^{n - S} \right) \right. \\
    & \cdot \left. \prod_{i = 0}^{S} (m - N - i)^{t \binom{S}{i} (n - 1)^{S - i} n^{n - S}} \right].
\end{align*}
\end{lem}

The next lemma is a natural extension of Lemma~\ref{lemma:paulconstruction} in the case where $n=2$ and the partite set of $G$ of size greater than 2 is not necessarily divisible by 4.  We omit its proof since the proof is merely simple computations.

\begin{lem}\label{lemma:balancedextension}
Let $t\in \mathbb{N}$, and let $G_0 = K_{2,4t}$ with bipartition $\{x_1, x_2\}$,$\{y_1, \hdots, y_{4t}\}$. Let $L_0$ be the $m$-assignment for $G_0$ defined in Lemma~\ref{lemma:paulconstruction}. For each $c\in [3]$, let $G_c = K_{2,4t+c}$ with bipartition $\{x_1, x_2\}$,$\{y_1, \hdots, y_{4t+c}\}$, and let $L_3$ be the $m$-assignment for $G_3$ given by 
\[L_3(y_k) = \begin{cases}
L_0(y_k) & k\in [4t] \\
[m-2]\cup \{m-1, m+1\} & k = 4t+1 \\
[m-2]\cup \{m, m+2\} & k = 4t+2\\
[m-2]\cup \{m-1, m+2\} & k = 4t+3.
\end{cases}
\] Further, for each $c\in [2]$, let $L_c$ be the $m$-assignment for $G_c$ satisfying $L_c(v) = L_3(v)$ for all $v\in V(G_c)$. Then,
\begin{align*}
        P(G_0, L_0)
        &= (m - 2)(m - 1)^{4t} + (m - 3)(m - 2)^{4t + 1} + 4(m - 2)^{2t + 1}(m - 1)^{2t} \\
        & + 4(m - 2)^t(m - 1)^{2t}m^t, \\
        P(G_1, L_1)
        &= (m - 2)(m - 1)^{4t + 1} + (m - 3)(m - 2)^{4t + 2} \\
        & + 2(2m - 3)(m - 2)^{2t + 1}(m - 1)^{2t} + 4(m - 2)^t(m - 1)^{2t + 1}m^t, \\
        P(G_2, L_2)
        &= (m - 2)(m - 1)^{4t + 2} + (m - 3)(m - 2)^{4t + 3} \\
        & + 4(m - 2)^{2t + 2}(m - 1)^{2t+1} + 2(2m^2-4m+1)(m - 2)^t(m - 1)^{2t}m^t, \quad \text{and} \\
        P(G_3, L_3)
        &= (m - 2)(m - 1)^{4t + 3} + (m - 3)(m - 2)^{4t + 4} \\
        & + 2(2m - 3)(m - 2)^{2t + 2}(m - 1)^{2t + 1} \\
        & + 2(2m^2 - 4m + 1)(m - 2)^t(m - 1)^{2t + 1}m^t.
    \end{align*}
\end{lem}

Suppose that $G = K_{2,n}$ and its partite set of size two is $\{x_1,x_2 \}$.  It is worth mentioning that we suspect that when $\tau(G) > m$ there is an $m$-assignment $L$ for $G$ with $|L(x_1) \cap L(x_2)| = m-2$ such that $P(G,L) < P(G,m)$.  In other words, if $P_{\ell}(G,m) < P(G,m)$, then we suspect that there is always an $m$-assignment barrier in which the lists corresponding to $x_1$ and $x_2$ have exactly $m-2$ colors in common.  Moreover, in this situation the list assignments described in Lemma~\ref{lemma:balancedextension} are candidates for such a barrier.  Our next technical lemma is a combinatorial lemma that will be used repeatedly in the remainder of the paper.

\begin{lem}\label{lemma:counting}
Let $m\in \mathbb{N}$ and $L_1$, $L_2$ be two sets such that $\vert L_1\vert = \vert L_2\vert = m$. Then, let $A_1 = L_1\cap L_2$, $A_2 = L_1 - L_2$, and $A_3 = L_2-L_1$. Suppose that $K\subseteq L_1\cup L_2$ with $\vert K \vert = m$. Let $a_1 = \vert A_1\vert$, $a_2 = \vert A_2\vert$, $a_3 = \vert A_3\vert$, $k_{1} = \vert K\cap A_1\vert$, $k_{2} = \vert K\cap A_2\vert$, and $k_{3} = \vert K\cap A_3\vert$. Finally, for any $(i,j)\in L_1\times L_2$, let $q_{(i,j)} = \vert K - \{i,j\}\vert$. Then, the following statements hold: 
\begin{enumerate}
    \item Let $B_1$ be the set of all $(i,i)\in A_1\times A_1$. Then, amongst all ordered pairs $b\in B_1$, exactly $k_1$ of them satisfy $q_{b} = m-1$, and exactly $a_1 - k_1$ satisfy $q_{b} = m$.
    \item Let $B_2$ be the set of all $(i,j)\in A_1\times A_1$ satisfying $i\neq j$. Then, amongst all ordered pairs $b\in B_2$, exactly $k_1(k_1-1)$ of them satisfy $q_b = m-2$, exactly $2k_1(a_1-k_1)$ of them satisfy $q_b = m-1$, and exactly $(a_1-k_1)(a_1-k_1-1)$ of them satisfy $q_b = m$.
    \item Let $B_3$ be the set of all $(i,j)\in A_1\times A_3$. Then, amongst all ordered pairs $b\in B_3$, exactly $k_1k_3$ of them satisfy $q_b = m-2$, exactly $k_1a_3 + k_3a_1 - 2k_1k_3$ of them satisfy $q_b = m-1$, and exactly $(a_1-k_1)(a_3-k_3)$ of them satisfy $q_b = m$.
    \item Let $B_4$ be the set of all $(i,j)\in A_1\times A_2$. Then, amongst all ordered pairs $b\in B_4$, exactly $k_1k_2$ of them satisfy $q_b = m-2$, exactly $k_1a_2 + k_2a_1 - 2k_1k_2$ of them satisfy $q_b = m-1$, and exactly $(a_1-k_1)(a_2-k_2)$ of them satisfy $q_b = m$.
    \item Let $B_5$ be the set of all $(i,j)\in A_2\times A_3$. Then, amongst all ordered pairs $b\in B_5$, exactly $k_2k_3$ of them satisfy $q_b = m-2$, exactly $k_2a_3 + k_3a_2 - 2k_2k_3$ of them satisfy $q_b = m-1$, and exactly $(a_2-k_2)(a_3-k_3)$ of them satisfy $q_b = m$.
\end{enumerate}
\end{lem}

\begin{proof}
Note $q_{(i,j)} = m - \vert K\cap \{i,j\}\vert$, which implies $q_{(i,j)}\in \{m, m-1, m-2\}$ for all $(i,j)\in L_1\times L_2$.

First, consider all $(i,i)\in B_1$, of which there are $\vert B_1\vert = a_1$. In this case, $\vert K\cap \{i\}\vert = 1$ if and only if $i\in K\cap A_1$, giving us exactly $k_1$ possible $(i,i)$ satisfying $q_{(i,i)} = m-1$. Moreover, since we clearly have $\vert K\cap \{i\}\vert \leq 1$, it follows that the number of $(i,i)\in B_1$ with the property $\vert K\cap \{i\}\vert = 0$ is exactly equal to $\vert B_1\vert - k_1$, from which statement 1 follows.

Next, consider all $(i,j)\in B_2$, of which there are $\vert B_2\vert = a_1(a_1-1)$. In this case, $\vert K\cap \{i,j\}\vert = 2$ if and only if $i,j\in K\cap A_1$; therefore, there are exactly $k_1(k_1-1)$ possible $(i,j)$ such that $q_{(i,j)} = m-2$. Similarly, $\vert K\cap \{i,j\}\vert = 0$ if and only if $i,j\notin K\cap A_1$, giving us exactly $(a_1-k_1)(a_1-k_1-1)$ possible $(i,j)$ satisfying $q_{(i,j)} = m$. It follows that the number of $(i,j)\in B_2$ with the property $\vert K\cap \{i,j\}\vert = 1$ is exactly equal to $\vert B_2\vert - (k_1(k_1-1) + (a_1-k_1)(a_1-k_1-1))$, from which statement 2 follows.

Finally, consider when $(i,j)\in \bigcup_{i=3}^5B_i$. Since $B_3$, $B_4$, and $B_5$ are pairwise disjoint, we assume without loss of generality that $i\in A_p$ and $j\in A_q$ for $p,q\in [3]$ and $p\neq q$. Under this condition, there are $a_pa_q$ such $(i,j)$. Now, we have that $\vert K\cap \{i,j\}\vert = 2$ if and only if $i\in K\cap A_p$ and $j\in K\cap A_q$; thus, there are exactly $k_pk_q$ possible $(i,j)$ satisfying $q_{(i,j)} = m-2$. Similarly, $\vert K\cap \{i,j\}\vert = 0$ if and only if $i\notin K\cap A_p$ and $j\notin K\cap A_q$, giving us exactly $(a_p-k_p)(a_q-k_q)$ possible $(i,j)$ such that $q_{(i,j)}=m-1$. It follows that the number of $(i,j)$ with $i\in A_p$, $j\in A_q$, and the property $\vert K\cap \{i,j\}\vert = 1$, is exactly equal to $a_pa_q - (k_pk_q + (a_p-k_p)(a_q-k_q))$, from which statements 3, 4, and 5 all follow, so we are done.
\end{proof}

Suppose that $G = K_{2,n}$ and its partite set of size two is $\{x_1,x_2 \}$.  Our final technical lemma gives a general lower bound on $P(G,L)$ when $L$ is an $m$-assignment for $G$ with $|L(x_1) \cap L(x_2)| = m-2$.  This lower bound will be quite useful when we prove Theorem~\ref{thm: casework}.

\begin{lem}\label{lemma:dm-2}
Let $m\in \mathbb{N}$, $m\geq 3$, and $G = K_{2,n}$ with bipartition $\{x_1, x_2\}$, $\{y_1, \hdots, y_n\}$. Let $L$ be an $m$-assignment for $G$ satisfying $\vert L(x_1)\cap L(x_2)\vert = m-2$ and $L(y_j)\subseteq L(x_1)\cup L(x_2)$ for all $j\in [n]$. Then, when $m > 3$, the following inequality holds for all ordered quadruples of nonnegative integers $(a_1, a_2, a_3, a_4)$ satisfying $a_1+a_2+a_3+a_4=n$:
\begin{align*}
    P(G,L)&\geq (m-2)\left[(m-1)^{(m-2)(a_1+a_2) + (m-3)a_3+(m-4)a_4}m^{a_3+2a_4}\right]^{1/(m-2)}\\
    &\qquad + 4(m-2)\left[(m-2)^{2(m-2)(a_1+a_2)+3(m-3)a_3+4(m-4)a_4}\right.\\
    &\qquad \left.(m-1)^{2(m-2)(a_1+a_2)+ma_3+8a_4}m^{a_3}\right]^{1/(4(m-2))} \\
    &\qquad + 4\left[(m-2)^{a_2+2a_3+4a_4}(m-1)^{4a_1+2a_2+2a_3}m^{a_2}\right]^{1/4} + Q(m),
\end{align*}
where
\begin{align*}
    Q(m) &= (m-2)(m-3)\left[(m-2)^{(m-2)(m-3)(a_1+a_2) + (m-4)((m-3)a_3 + (m-5)a_4)}\right.\\
    &\qquad \left.(m-1)^{2(m-3)a_3+4(m-4)a_4}m^{2a_4}\right]^{1/((m-2)(m-3))}.
\end{align*}
When $m=3$, the same inequality holds if we instead let $Q(m)=0$ and impose the additional constraint that $a_4=0$.
\end{lem}

\begin{proof}
For each $(i,j)\in L(x_1)\times L(x_2)$ and $k\in [n]$, let $q_{k, (i,j)} = \vert L(y_k) - \{i,j\}\vert$. 

Without loss of generality, assume $L(x_1) = [m]$ and $L(x_2) = [m-2]\cup\{m+1, m+2\}$. Then $L(y_k)\in \binom{[m+2]}{m}$ for each $k\in [n]$. For each $A\in \binom{[m+2]}{m}$, let $z_{A} = \vert \{k\in [n]: L(y_k) = A\}\vert$. 

Let $B = L(x_1)-L(x_2) = \{m-1,m\}$ and $C = L(x_2)-L(x_1) = \{m+1,m+2\}$. Then, let $\{\mathcal{X}, \mathcal{Y}, \mathcal{Z}, \mathcal{W}\}$ be the partition of $\binom{[m+2]}{m}$ satisfying
\begin{align*}
    \mathcal{X} &= \{[m-2]\cup B,[m-2]\cup C\} = \{L(x_1),L(x_2)\}, \\
    \mathcal{Y} &= \{Y: ([m-2] \subset Y)\}\cap \{Y: \vert Y\cap B\vert = 1\}, \\
    \mathcal{Z} &= \{Z: \vert Z\cap [m-2] \vert = m-3\}, \\
    \mathcal{W} &= \{W: \vert W\cap [m-2] \vert = m-4\},
\end{align*}
where if $m=3$ we let $\mathcal{W} = \emptyset$. Broadly speaking, this partition serves to separate the set of possible lists for each $y_k$ into four distinct list types. Let $b_1 = \sum_{S\in \mathcal{X}}z_S$, $b_2 = \sum_{S\in \mathcal{Y}}z_S$, $b_3 = \sum_{S\in \mathcal{Z}}z_S$, and $b_4 = \sum_{S\in \mathcal{W}}z_S$. Notice that $b_1+b_2+b_3+b_4=n$. 

Our general approach will be to compute lower bounds for the values of $\vert \mathcal{C}_{(i,j)}\vert$ across all $(i,j)\in L(x_1)\times L(x_2)$. More specifically, for each $(i,j)\in L(x_1)\times L(x_2)$, we analyze the value of $q_{k, (i,j)}$ for each $y_k$ by performing casework on each of the four possible list types as described in the previous paragraph. Note that we necessarily have $q_{k, (i,j)} \in \{m-2,m-1,m\}$.

Consider when $(i,j)\in E$ where $E = \{(a,a): a\in [m-2]\}$. When $L(y_k) \in \mathcal{X}$, Statement 1 in Lemma~\ref{lemma:counting} implies $\vert \{a\in [m-2]: q_{k, (a,a)} = m-1\}\vert = m-2$. By repeatedly applying Statement 1 in Lemma~\ref{lemma:counting} in the same way for each of the three remaining list types, we get that: if $L(y_k)\in \mathcal{Y}$, then $\vert \{a\in [m-2]: q_{k, (a,a)} = m-1\}\vert = m-2$; if $L(y_k)\in \mathcal{Z}$, then $\vert \{a\in [m-2]: q_{k, (a,a)} = m-1\}\vert = m-3$ and $\vert \{a\in [m-2]: q_{k, (a,a)} = m\}\vert = 1$; and, if $L(y_k)\in \mathcal{W}$, then $\vert \{a\in [m-2]: q_{k, (a,a)} = m-1\}\vert = m-4$ and $\vert \{a\in [m-2]: q_{k, (a,a)}=m\}\vert = 2$. In sum, our casework implies that $\sum_{(a,a)\in E} \vert \{y_k: q_{k, (a,a)} = m-1\}\vert = (m-2)(b_1+b_2) + (m-3)b_3 + (m-4)b_4$, and $\sum_{(a,a)\in E} \vert \{y_k: q_{k, (a,a)} = m\}\vert = b_3 + 2b_4$. So, by the AM-GM inequality:
\begin{equation}\label{eq:dm-2amgm1}
\sum_{(a,a)\in E}\vert\mathcal{C}_{(a,a)}\vert \geq (m-2)\left[(m-1)^{(m-2)(b_1+b_2)+(m-3)b_3+(m-4)b_4}m^{b_3+2b_4}\right]^{1/(m-2)}.
\end{equation}

Next, consider when $(i,j) \in F$ where $F = \{(a,b): a,b\in [m-2], a \neq b\}$. Note that when $m=3$, $\vert F\vert = 0$, which implies that $\sum_{(i,j)\in F}\vert \mathcal{C}_{(i,j)}\vert = 0$. So, assume $m\geq 4$. When $L(y_k) \in \mathcal{X}$, we may apply Statement 2 in Lemma~\ref{lemma:counting} to find that $\vert \{(i,j)\in F: q_{k, (i,j)} = m-2\}\vert = (m-2)(m-3)$. By repeatedly applying Statement 2 in Lemma~\ref{lemma:counting} in the same way for each of the three remaining list types, we get that: if $L(y_k)\in \mathcal{Y}$, then $\vert \{(i,j)\in F: q_{k, (i,j)} = m-2\} \vert = (m-2)(m-3)$; if $L(y_k)\in \mathcal{Z}$, then $\vert \{(i,j)\in F: q_{k, (i,j)} = m-2\}\vert = (m-3)(m-4)$ and $\vert \{(i,j)\in F: q_{k, (i,j)} = m-1\}\vert = 2(m-3)$; and, if $L(y_k)\in \mathcal{W}$, then $\vert \{(i,j)\in F: q_{k, (i,j)} = m-2\}\vert = (m-4)(m-5)$, $\vert \{(i,j)\in F: q_{k, (i,j)} = m-1\}\vert = 4(m-4)$, and $\vert \{(i,j)\in F: q_{k, (i,j)} = m\}\vert = 2$. In sum, our casework implies that $\sum_{(i,j)\in F} \vert \{y_k: q_{k, (i,j)} = m-2\}\vert = (m-2)(m-3)(b_1+b_2) + (m-4)((m-3)b_3+(m-5)b_4)$, $\sum_{(i,j)\in F} \vert \{y_k: q_{k, (i,j)} = m-1\}\vert = 2(m-3)b_3+4(m-4)b_4$, and $\sum_{(i,j)\in F} \vert \{y_k: q_{k, (i,j)} = m\}\vert = 2b_4$. So, by the AM-GM inequality, we have that for $m \geq 4$:
\begin{align}\label{eq:dm-2amgm2}
\sum_{(i,j)\in F} \vert\mathcal{C}_{(i,j)}\vert &\geq (m-2)(m-3)\left[(m-2)^{(m-2)(m-3)(b_1+b_2) + (m-4)((m-3)b_3+(m-5)b_4)}\right. \notag\\
&\qquad \left.(m-1)^{2(m-3)b_3+4(m-4)b_4}m^{2b_4}\right]^{1/((m-2)(m-3))},
\end{align}

Next, consider when $(i,j) \in H$ where $H = ([m-2] \times C) \cup (B \times [m-2])$. In this case, when $L(y_k) \in \mathcal{X}$, we may apply Statements 3 and 4 in Lemma~\ref{lemma:counting} to find that $\vert \{(i,j)\in H: q_{k, (i,j)} = m-2\}\vert = 2(m-2)$ and $\vert \{(i,j)\in H: q_{k, (i,j)} = m-1\}\vert = 2(m-2)$. By repeatedly applying Statements 3 and 4 in Lemma~\ref{lemma:counting} in the same way for each of the three remaining list types, we get that: if $L(y_k)\in \mathcal{Y}$, then $\vert \{(i,j)\in H: q_{k, (i,j)} = m-2\} \vert = 2(m-2)$ and $\vert \{(i,j)\in H: q_{k, (i,j)} = m-1\} \vert = 2(m-2)$; if $L(y_k)\in \mathcal{Z}$, then $\vert \{(i,j)\in H: q_{k, (i,j)} = m-2\}\vert = 3(m-3)$, $\vert \{(i,j)\in H: q_{k, (i,j)} = m-1\}\vert = m$, and $\vert \{(i,j)\in H: q_{k, (i,j)} = m\}\vert = 1$; and, if $L(y_k)\in \mathcal{W}$, then $\vert \{(i,j)\in H: q_{k, (i,j)} = m-2\}\vert = 4(m-4)$ and $\vert \{(i,j)\in H: q_{k, (i,j)} = m-1\}\vert = 8$. In sum, our casework implies that $\sum_{(i,j)\in H} \vert \{y_k: q_{k, (i,j)} = m-2\}\vert = 2(m-2)(b_1+b_2) + 3(m-3)b_3+4(m-4)b_4$, $\sum_{(i,j)\in H} \vert \{y_k: q_{k, (i,j)} = m-1\}\vert = 2(m-2)(b_1+b_2)+mb_3+8b_4$, and $\sum_{(i,j)\in H} \vert \{y_k: q_{k, (i,j)} = m\}\vert = b_3$. So, by the AM-GM inequality:
\begin{align}\label{eq:dm-2amgm3}
\sum_{(i,j)\in H} \vert\mathcal{C}_{(i,j)}\vert &\geq 4(m-2)\left[(m-2)^{2(m-2)(b_1+b_2) + 3(m-3)b_3+4(m-4)b_4}\right. \notag\\
&\qquad \left.(m-1)^{2(m-2)(b_1+b_2)+mb_3+8b_4}m^{b_3}\right]^{1/(4(m-2))}.
\end{align}

Finally, consider when $(i,j) \in J$ where $J = B \times C$. In this case, when $L(y_k) \in \mathcal{X}$, we may apply Statement 5 in Lemma~\ref{lemma:counting} to find that $\vert \{(i,j)\in J: q_{k, (i,j)} = m-1\}\vert = 4$. By repeatedly applying Statement 5 in Lemma~\ref{lemma:counting} in the same way for each of the three remaining list types, we get that: if $L(y_k)\in \mathcal{Y}$, then $\vert \{(i,j)\in J: q_{k, (i,j)} = m-2\} \vert = 1$, $\vert \{(i,j)\in J: q_{k, (i,j)} = m-1\} \vert = 2$, and $\vert \{(i,j)\in J: q_{k, (i,j)} = m\} \vert = 1$; if $L(y_k)\in \mathcal{Z}$, then $\vert \{(i,j)\in J: q_{k, (i,j)} = m-2\}\vert = 2$ and $\vert \{(i,j)\in J: q_{k, (i,j)} = m-1\}\vert = 2$; and, if $L(y_k)\in \mathcal{W}$, then $\vert \{(i,j)\in J: q_{k, (i,j)} = m-2\}\vert = 4$. In sum, our casework implies that $\sum_{(i,j)\in J} \vert \{y_k: q_{k, (i,j)} = m-2\}\vert = b_2+2b_3+4b_4$, $\sum_{(i,j)\in J} \vert \{y_k: q_{k, (i,j)} = m-1\}\vert = 4b_4+2b_2+2b_3$, and $\sum_{(i,j)\in J} \vert \{y_k: q_{k, (i,j)} = m\}\vert = b_2$. So, by the AM-GM inequality:
\begin{align}\label{eq:dm-2amgm4}
\sum_{(i,j)\in J} \vert\mathcal{C}_{(i,j)}\vert \geq 4\left[(m-2)^{b_2+2b_3+4b_4}(m-1)^{4b_1+2b_2+2b_3}m^{b_2}\right]^{1/4}.
\end{align}
Finally, since
\begin{align*}
P(G,L) &= \sum_{(i,j)\in L(x_1)\times L(x_2)} \vert\mathcal{C}_{(i,j)}\vert \\
&= \sum_{(i,j)\in E} \vert\mathcal{C}_{(i,j)}\vert + \sum_{(i,j)\in F} \vert\mathcal{C}_{(i,j)}\vert + \sum_{(i,j)\in H} \vert\mathcal{C}_{(i,j)}\vert + \sum_{(i,j)\in J} \vert\mathcal{C}_{(i,j)}\vert,
\end{align*}
summing Inequalities~\ref{eq:dm-2amgm1}, \ref{eq:dm-2amgm2}, \ref{eq:dm-2amgm3}, and \ref{eq:dm-2amgm4} finishes the proof.
\end{proof}

Our last ingredient needed for the proof of Theorem~\ref{thm: casework} is the following theorem.

\begin{thm}[\cite{KK22}]\label{theorem:firstthresholdpaper}
Let $G = K_{2,n}$. Then, $P_{\ell}(G,m)<P(G,m)$ whenever \[\left\lfloor \frac{n}{4}\right\rfloor \geq (m-1)^2\ln{(16/7)}.\]
\end{thm}

We are now ready to prove Theorem~\ref{thm: casework}.  Since the proofs of the three statements of Theorem~\ref{thm: casework} are similar, we will prove only Statement~(i).  The proofs of the other statements will be in Appendix~\ref{finish}.

\begin{proof}
Let $G = K_{2, n}$ with bipartition $\{x_1, x_2\}, \{y_1, \ldots, y_n\}$.  Since $\tau(C_4) = 2$ and we wish to begin by proving the first part of Statement~(i), we will assume that $3 \leq n \leq 10$.  Let $L$ be a 3-assignment for $G$ such that $L(y_j) \subseteq L(x_1) \cup L(x_2)$ for all $j \in [n]$ and $P(G, L) = P_{\ell}(G, 3)$ (we know such an $L$ exists by Lemma~\ref{lemma:unionbadgeneral}). For each $(i,j)\in L(x_1)\times L(x_2)$ and $k\in [n]$, let $q_{k,(i,j)} = \vert L(y_k) - \{i,j\}\vert$. Let $B = L(x_1) - L(x_2)$, $C = L(x_2) - L(x_1)$, $D = L(x_1) \cap L(x_2)$, and $d=\vert D\vert$. Clearly, $d\in \{0,1,2,3\}$. We will show that $P(G, L) \geq P(G, 3) = 3\cdot 2^n + 6$ for each possible $d$. From this, it follows that $P_{\ell}(G, 3) = P(G, 3)$.

If $d = 3$, then the fact that $L(y_j) \subseteq L(x_1) \cup L(x_2)$ for all $j \in [n]$ implies that $L$ assigns the same list to every vertex in $G$ which means $P(G, L) = P(G, 3)$. If $d = 2$, then by Lemma~\ref{lemma:dm-1bad} we have $P(G, L) \geq P(G, 3)$. If $d=1$, then by Lemma~\ref{lemma:dm-2} we have 
\begin{align*}
    P(G,L)\geq 2^{a_1}2^{a_2}3^{a_3} + 4\cdot 2^{a_1/2}2^{a_2/2}24^{a_3/4} + 4\cdot 2^{a_1}12^{a_2/4}2^{a_3/2}
\end{align*}
for all ordered triples of nonnegative integers $(a_1, a_2, a_3)$ satisfying $a_1+a_2+a_3=n$. It is easy to verify computationally that the expression above is at least $P(G,3)$ over all such ordered triples, for all $3\leq n\leq 10$. See the code in the first program in Appendix~\ref{code}. 

So, suppose that $d = 0$. Without loss of generality, assume that $L(x_1) = \{1, 2, 3\}$ and $L(x_2) = \{4, 5, 6\}$. Then $L(y_k) \in \binom{[6]}{3}$ for each $k \in [n]$. For each $A \in \binom{[6]}{3}$, let $z_A = \lvert \{k \in [n] : L(y_k) = A\} \rvert$. Notice that $\sum_{A \in
\binom{[6]}{3}} z_A = n$. 

Now, let $\{\mathcal{X}, \mathcal{Y}\}$ be the partition of $\binom{[6]}{3}$ satisfying 
\begin{align*}
\mathcal{X} &= \{\{1,2,3\},\{4,5,6\}\},\text{ and} \\
\mathcal{Y} &= \binom{[6]}{3} - \mathcal{X}.
\end{align*}
Let $a_1 = \sum_{S\in \mathcal{X}}z_S$ and $a_2 = \sum_{S\in \mathcal{Y}}z_S$.

Now, we compute a lower bound for $\vert \mathcal{C}_{(i,j)}\vert$ across all $(i,j)\in L(x_1)\times L(x_2)$. Note that we necessarily have $q_{k, (i,j)}\in \{1,2,3\}$. When $L(y_k)\in \mathcal{X}$, then we have $\vert L(y_k)\cap B\vert = 3$ or $\vert L(y_k)\cap B\vert = 0$. For such $k$, applying all five statements in Lemma~\ref{lemma:counting} gives $\vert \{(i,j)\in L(x_1)\times L(x_2): q_{k,(i,j)}=2\}\vert = 9$. When $L(y_k)\in \mathcal{Y}$, then we have $\vert L(y_k)\cap B\vert = 2$ or $\vert L(y_k)\cap B\vert = 1$. For such $k$, applying all five statements in Lemma~\ref{lemma:counting} in the same manner gives $\vert \{(i,j)\in L(x_1)\times L(x_2): q_{k,(i,j)}=1\}\vert = 2$, $\vert \{(i,j)\in L(x_1)\times L(x_2): q_{k,(i,j)}=2\}\vert = 5$, and $\vert \{(i,j)\in L(x_1)\times L(x_2): q_{k,(i,j)}=3\}\vert = 2$. So, by the AM-GM inequality:
\begin{align*}
    P(G, L) 
    &= \sum_{(i, j) \in L(x_1) \times L(x_2)} \vert\mathcal{C}_{(i, j)}\vert \\
    &\geq 9\cdot (2^{9a_1+5a_2}3^{2a_2})^{1/9} \\
    &= 9 \cdot 2^{a_1} 288^{a_2/9} \geq 9\cdot 288^{n/9}.
\end{align*}
It is easy to computationally verify that the expression above is greater than or equal to $P(G,3)$ for all $3\leq n\leq 10$. Since we have exhausted all possible values for $d$, the proof of the first part of Statement~(i) is complete.

Now, we turn our attention to the second part of Statement~(i).  If $n \geq 16$, then Theorem~\ref{theorem:firstthresholdpaper} implies $P_{\ell}(K_{2,n},3) < P(K_{2,n},3)$ since $\lfloor n/4\rfloor \geq 4\ln(16/7)$.  On the other hand, if $n \in \{12, 13, 14, 15\}$, it is easy to verify from Lemma~\ref{lemma:balancedextension} that there is a 3-assignment $L$ for $G$ with the property that $P(G,L) < P(G,3)$. 
\end{proof}

\section{A General Upper Bound} \label{two}

Notice that in our proof of Theorem~\ref{thm: casework}, the most delicate parts of the argument rely on several careful applications of the AM-GM inequality.  Our idea for proving a general upper bound on $\tau(K_{2,n})$ is to see what happens if we insist on only applying the AM-GM inequality once.  Our first lemma makes this precise.

\begin{lem}\label{lemma:oneamgm}
Let $m, n \in \mathbb{N}$ with $m \geq 4$ and $n \geq 3$. Let 
\[f_i(m) = \ln\left(1-\frac{i}{m}\right)-\frac{1}{m^2}\left(2(m-1)\ln\left(1-\frac{1}{m}\right)+(m-1)^2\ln\left(1-\frac{2}{m}\right)\right)\]
for each $i \in [2]$. Suppose $\epsilon \in \mathbb{R}$ satisfies $0 < \epsilon < 1 - 1 / m$. If 
\[\frac{\ln(\epsilon) + \ln(m) - \ln(m - 1)}{f_2(m)} \leq n \leq \frac{\ln(1 - \epsilon) + \ln(m)}{f_1(m)},\]
then $P_{\ell}(K_{2, n}, m) = P(K_{2, n}, m)$.
\end{lem}

\begin{proof}
Let $G = K_{2, n}$ with bipartition $\{x_1, x_2\}, \{y_1, \ldots, y_n\}$. Let $L$ be an $m$-assignment for $G$ such that $L(y_k) \subseteq L(x_1) \cup L(x_2)$ and $P(G, L) = P_{\ell}(G, m)$ (we know such an $L$ exists by Lemma~\ref{lemma:unionbadgeneral}). Let $A = L(x_1) \cap L(x_2)$, $B = L(x_1) - L(x_2)$, and $C = L(x_2) - L(x_1)$, and let $d = \lvert A \rvert$, $b = \vert B\vert$, $c = \vert C\vert$. Let $a_k = \vert L(y_k) \cap A\vert, b_k = \vert L(y_k) \cap B\vert$, and $c_k = \vert L(y_k) \cap C\vert$ for each $k \in [n]$. Notice that for any $k \in [n]$, $a_k + b_k + c_k = d + b = d + c = m$.

If $d = m$, then clearly $P_{\ell}(G, m) = P(G, L) = P(G, m)$. If $d = m - 1$, then by Lemma~\ref{lemma:dm-1bad} we have $P_{\ell}(G, m) = P(G, L) \geq P(G, m)$, which implies that $P_{\ell}(G, m) = P(G, m)$. Hence, for the rest of this proof, we assume that $d \leq m - 2$.

For each $k \in [n]$ and $(i, j) \in L(x_1) \times L(x_2)$, let $q_{k, (i, j)} = \lvert L(y_k) - \{i, j\} \rvert$. Then, by the AM-GM inequality,
\begin{align*}
    P(G, L)
    &= \sum_{(i, j) \in L(x_1) \times L(x_2)} \prod_{k = 1}^n q_{k, (i, j)} \\
    &\geq m^2 \left( \prod_{(i, j) \in L(x_1) \times L(x_2)} \prod_{k = 1}^n q_{k, (i, j)} \right)^{1/m^2} \\
    &= m^2 \left( \prod_{k = 1}^n \prod_{(i, j) \in L(x_1) \times L(x_2)} q_{k, (i, j)} \right)^{1/m^2}.
\end{align*}

Consider some $k \in [n]$ and $(i, j) \in L(x_1) \times L(x_2)$. Let $u_k = a_k + b_k$ and $w_k = a_k + c_k$. Clearly, $q_{k, (i, j)} \in \{m - 2, m - 1, m\}$. By Lemma~\ref{lemma:counting}, 
\begin{align*}
    \lvert \{(i, j) \in L(x_1) \times L(x_2) : q_{k, (i, j)} = m - 2\} \rvert &= a_k^2-a_k+a_kc_k+a_kb_k+b_kc_k \\
    &= u_kw_k - a_k, \\
    \lvert \{(i, j) \in L(x_1) \times L(x_2) : q_{k, (i, j)} = m - 1\} \rvert &= a_k+2a_k(d-a_k) + (a_kc+c_kd-2a_kc_k) + \\
    &\qquad (a_kb+b_kd-2a_kb_k) + (b_kc+c_kb-2b_kc_k) \\
    &= a_k + u_k(m - w_k) + w_k(m - u_k),\text{ and} \\
    \lvert \{(i, j) \in L(x_1) \times L(x_2) : q_{k, (i, j)} = m \} \rvert &= (d-a_k) + (d-a_k)(d-a_k-1) \\ 
    &\qquad + (d-a_k)(c-c_k) + (d-a_k)(b-b_k) \\
    &\qquad + (b-b_k)(c-c_k) \\
    &= (m - u_k)(m - w_k).
\end{align*}

Therefore,
\begin{align*}
    P(G, L)
    &\geq m^2 \left( \prod_{k = 1}^n \prod_{(i, j) \in L(x_1) \times L(x_2)} q_{k, (i, j)} \right)^{1 / m^2} \\
    &= m^2 \left( \prod_{k = 1}^n (m - 2)^{u_kw_k - a_k} (m - 1)^{u_k(m - w_k) + w_k(m - u_k) + a_k} m^{(m - u_k)(m - w_k)} \right)^{1 / m^2}.
\end{align*}

For each $k \in [n]$,
\begin{align*}
    (m - 2)&^{u_kw_k - a_k}(m - 1)^{u_k(m - w_k) + w_k(m - u_k) + a_k}m^{(m - u_k)(m - w_k)} \\
    &= (m - 2)^{u_kw_k}(m - 1)^{u_k(m - w_k) + w_k(m - u_k)}m^{(m - u_k)(m - w_k)} \left( \frac{m - 1}{m - 2} \right)^{a_k} \\
    &= (m - 2)^{u_kw_k}(m - 1)^{u_k(m - w_k)}(m - 1)^{mw_k}(m - 1)^{-u_kw_k} \\
    & \cdot m^{m^2}m^{-mw_k}m^{-u_k(m - w_k)} \left( \frac{m - 1}{m - 2} \right)^{a_k} \\
    &= m^{m^2} \left( \frac{m - 2}{m - 1} \right)^{u_kw_k} \left( \frac{m - 1}{m} \right)^{u_k(m - w_k)} \left( \frac{m - 1}{m} \right)^{mw_k} \left( \frac{m - 1}{m - 2} \right)^{a_k} \\
    &= m^{m^2} \left( \frac{m(m - 2)}{(m - 1)^2} \right)^{u_kw_k} \left( \frac{m - 1}{m} \right)^{m(u_k + w_k)} \left( \frac{m - 1}{m - 2} \right)^{a_k} \\
    &= m^{m^2} \left( \frac{m(m - 2)}{(m - 1)^2} \right)^{(m - c_k)(m - b_k)} \left( \frac{m - 1}{m} \right)^{m(m + a_k)} \left( \frac{m - 1}{m - 2} \right)^{a_k}.
\end{align*}
Now, note that $d = \vert A\vert \geq \vert L(y_k)\cap A\vert = a_k$. Also, by the AM-GM inequality, $(m - c_k)(m - b_k) \leq ((2m - c_k - b_k)/2)^2 = ((m+d)/2)^2$. Thus, in conjunction with the fact that $(m(m-2)/(m-1)^2) < 1$ and $(m-1)/m < 1$, we obtain the following lower bound on $(m - 2)^{u_kw_k - a_k}(m - 1)^{u_k(m - w_k) + w_k(m - u_k) + a_k}m^{(m - u_k)(m - w_k)}$:
\begin{align*}
    &m^{m^2} \left( \frac{m(m - 2)}{(m - 1)^2} \right)^{((m + d) / 2)^2} \left( \frac{m - 1}{m} \right)^{m(m + d)} \\
    &= m^{((m + d) / 2)^2 - md} (m - 1)^{m(m + d) - (m + d)^2 / 2} (m - 2)^{((m + d) / 2)^2} \\
    &= m^{((m - d) / 2)^2} (m - 1)^{(m^2 - d^2) / 2} (m - 2)^{((m + d) / 2)^2}.
\end{align*}
Hence,
\begin{align*}
    P(G, L)
    &\geq m^2 \left( \prod_{k = 1}^n (m - 2)^{u_kw_k - a_k} (m - 1)^{u_k(m - w_k) + w_k(m - u_k) + a_k} m^{(m - u_k)(m - w_k)} \right)^{1 / m^2} \\
    &\geq m^2 \left( \prod_{k = 1}^n m^{((m - d) / 2)^2} (m - 1)^{(m^2 - d^2) / 2} (m - 2)^{((m + d) / 2)^2} \right)^{1 / m^2} \\
    &= m^2 \left( m^{n((m - d) / 2)^2} (m - 1)^{n(m^2 - d^2) / 2} (m - 2)^{n((m + d) / 2)^2} \right)^{1 / m^2} \\
    &= m^2 m^{(n / m^2)((m - d) / 2)^2} (m - 1)^{(n / m^2)(m^2 - d^2) / 2} (m - 2)^{(n / m^2) ((m + d) / 2)^2}.
\end{align*}

Recall that $0 \leq d \leq m - 2$. We will now show that the final expression above is smallest when $d = m - 2$. Notice that, for $0 \leq d \leq m - 3$, we have
\begin{align*}
    &\frac{\displaystyle m^2m^{(n / m^2)((m - (d + 1)) / 2)^2} (m - 1)^{(n / m^2)(m^2 - (d + 1)^2) / 2} (m - 2)^{(n / m^2)((m + (d + 1)) / 2)^2}}{\displaystyle m^2 m^{(n / m^2)((m - d) / 2)^2} (m - 1)^{(n / m^2)(m^2 - d^2) / 2} (m - 2)^{(n / m^2)((m + d) / 2)^2}} \\
    &= m^{(n / m^2)(1 / 4 - (m - d) / 2)} (m - 1)^{-(n / m^2)(d + 1 / 2)} (m - 2)^{(n / m^2)(1 / 4 + (m + d) / 2)} \\
    &= \left( m^{1 / 4 - (m - d) / 2} (m - 1)^{-(d + 1 / 2)} (m - 2)^{1 / 4 + (m + d) / 2} \right)^{n / m^2} \\
    &= \left( m^{1 - 2(m - d)} (m - 1)^{-2(2d + 1)} (m - 2)^{1 + 2(m + d)} \right)^{(1 / 4)(n / m^2)} \\
    &= \left( m^{-2m + 2d + 1} (m - 1)^{-2(2d + 1)} (m - 2)^{2m + 2d + 1} \right)^{(1 / 4)(n / m^2)} \\
    &= \left( \left( \frac{m - 2}{m} \right)^{2m} \left( \frac{m(m - 2)}{(m - 1)^2} \right)^{2d + 1} \right)^{(1 / 4)(n / m^2)} \\
    &< 1.
\end{align*}
Therefore,
\begin{align*}
    P(G, L)
    &\geq m^2 m^{(n / m^2)((m - d) / 2)^2} (m - 1)^{(n / m^2)(m^2 - d^2) / 2} (m - 2)^{(n / m^2)((m + d) / 2)^2} \\
    &\geq m^2 m^{(n / m^2)((m - (m - 2)) / 2)^2} (m - 1)^{(n / m^2)(m^2 - (m - 2)^2) / 2} (m - 2)^{(n / m^2)((m + (m - 2)) / 2)^2} \\
    &= m^2 \left(m (m - 1)^{2(m - 1)} (m - 2)^{(m - 1)^2} \right)^{n / m^2}.
\end{align*}

We will now show that $P(G, L) \geq P(G, m)$. That is,
\begin{align*}
    m^2 \left( m (m - 1)^{2(m - 1)} (m - 2)^{(m - 1)^2} \right)^{n / m^2} \geq m(m - 1)^n + m(m - 1)(m - 2)^n.
\end{align*}
Notice that the above inequality is equivalent to the statement
\begin{align*}
    1
    &\geq \frac{m(m - 1)^n}{m^2 \left( m (m - 1)^{2(m - 1)}(m - 2)^{(m - 1)^2} \right)^{n / m^2}} + \frac{m(m - 1)(m - 2)^n}{m^2 \left( m (m - 1)^{2(m - 1)}(m - 2)^{(m - 1)^2} \right)^{n / m^2}} \\
    &= m^{-1 - n / m^2} (m - 1)^{n \left( 1 - (2 / m^2)(m - 1) \right)} (m - 2)^{-(n / m^2)(m - 1)^2} \\
    &+ m^{-1 - n / m^2} (m - 1)^{1 - 2(n / m^2)(m - 1)} (m - 2)^{n \left( 1 - (1 / m^2)(m - 1)^2 \right)}.
\end{align*}
Notice that the above inequality is satisfied if
\begin{align} \label{1-minus-epsilon}
    m^{-1 - n / m^2} (m - 1)^{n\left( 1 - (2 / m^2)(m - 1) \right)} (m - 2)^{-(n / m^2)(m - 1)^2} \leq 1 - \epsilon
\end{align}
and
\begin{align} \label{epsilon}
    m^{-1 - n / m^2} (m - 1)^{1 - 2(n / m^2)(m - 1)} (m - 2)^{n \left( 1 - (1 / m^2)(m - 1)^2 \right)} \leq \epsilon.
\end{align}
Inequality~\ref{1-minus-epsilon} is true if and only if
\begin{align*}
    n \cdot f_1(m) \leq \ln(1 - \epsilon) + \ln(m),
\end{align*}
and Inequality~\ref{epsilon} is true if and only if
\begin{align*}
    n \cdot f_2(m) \leq \ln(\epsilon) + \ln(m) - \ln(m - 1).
\end{align*}
Along with the fact that $f_2(m) < 0$, these two inequalities imply our desired result.
\end{proof}

From Lemma~\ref{lemma:oneamgm} we can deduce the following corollary which is a bit more practical to use.

\begin{cor}\label{corollary:oneamgmbounds}
Let $m, n \in \mathbb{N}$ with $m \geq 4$ and $n \geq 3$. Suppose $\epsilon \in \mathbb{R}$ satisfies $0 < \epsilon < 1 - 1 / m$. If 
\[ -\frac{m^2}{2}\ln{\left(\frac{m\epsilon}{m-1}\right)}\leq n \leq (m+1/2)\ln{(m(1-\epsilon))},\] 
then $P_{\ell}(K_{2, n}, m) = P(K_{2, n}, m)$.
\end{cor}

\begin{proof}
It suffices to show that our lower bound for $n$ is bounded below by the lower bound for $n$ in the statement of Lemma~\ref{lemma:oneamgm}, and that our upper bound for $n$ is bounded above by the upper bound in the statement of Lemma~\ref{lemma:oneamgm}.
 
We start by proving 
\begin{equation}\label{ineq:lb}
-\frac{m^2}{2}\ln{\left(\frac{m\epsilon}{m-1}\right)}\geq \frac{\ln(\epsilon) + \ln(m) - \ln(m - 1)}{f_2(m)}.
\end{equation}
When $x > -1$ and $x\neq 0$, it is well known that $x/(x+1) < \ln{(1+x)} < x$. Using these inequalities, we have that 
\[m^2f_2(m) = \ln\left(1-\frac{2}{m}\right)\left(2m-1\right)-\ln\left(1-\frac{1}{m}\right)\left(2m-2\right)\leq -\frac{2(2m-1)}{m} + 2 < 0\] since $m\geq 4$. Thus, it suffices to show that $-(m^2/2)f_2(m)\geq 1$, since $\ln{(m\epsilon / (m-1))} < 0$. This reduces to showing that 
\[\ln{\left(\left(1-\frac{2}{m}\right)^{1/2-m}\left(1-\frac{1}{m}\right)^{m-1}\right)}\geq 1,\]
which follows from
\begin{align*}
    \ln{\left(\left(1-\frac{2}{m}\right)^{1/2-m}\left(1-\frac{1}{m}\right)^{m-1}\right)} = \ln{\left(\left(\frac{m}{m-2}\right)^{1/2}\left(\frac{m-1}{m-2}\right)^{m-1}\right)} > 1.
\end{align*}
Thus, Inequality~\ref{ineq:lb} holds. Next, we show that
\begin{equation}\label{ineq:ub}
\left(m+\frac{1}{2}\right)\ln{(m(1-\epsilon))} \leq \frac{\ln(1 - \epsilon) + \ln(m)}{f_1(m)}.
\end{equation}
First note that 
\[m^2f_1(m) = (m-1)^2\ln{\left(1 + \frac{1}{m-2}\right)} + \ln{\left(1-\frac{1}{m}\right)} > m-1 - \frac{1}{m-1} > 0\]
since $m\geq 4$. Thus, it suffices to show that $(m+1/2)f_1(m)\leq 1$, since $\ln{(m(1-\epsilon))} > 0$. This reduces to showing that
\begin{equation}\label{upperboundineq}
\left(m+\frac{1}{2}\right)\left(\left(\frac{m-1}{m}\right)^{2}\ln\left(1+\frac{1}{m-2}\right)+\frac{1}{m^{2}}\ln\left(1-\frac{1}{m}\right)\right)\leq 1.
\end{equation}
Recall $\ln{(1+x)} = \sum_{n=1}^{\infty}(-1)^{n+1}x^n/n$ whenever $\vert x\vert < 1$. Rewriting, this is equivalent to $\ln{(1+x)}= (x - x^2/2 + x^3/3) + \sum_{k=2}^{\infty}(-x^{2k}/(2k) + x^{2k+1}/(2k+1)) = (x - x^2/2 + x^3/3) - \sum_{k=2}^{\infty}(x^{2k}/(2k(2k+1))\left(2k(1-x)+1\right))$. Since each term inside of the summand is positive, this implies the bound $\ln{(1+x)} < x-x^2/2+x^3/3$. Using this bound, and the previously used $\ln{(1+x)} < x$, for $\ln{(1+1/(m-2))}$ and $\ln{(1-1/m)}$ in Inequality~\ref{upperboundineq}, respectively, shows that 
\begin{align}\label{ineq:sixthpower}
\left(m+\frac{1}{2}\right)\left(\left( \frac{m-1}{m}\right)^{2}\ln\left(1+\frac{1}{m-2}\right)+\frac{1}{m^{2}}\ln\left(1-\frac{1}{m}\right)\right)
\leq \notag\\
\frac{\left(2m+1\right)\left(m^{5}-7m^{4}+17m^{3}-13m^{2}-5m+8\right)}{2\left(m-2\right)^{3}m^{3}}.
\end{align}
So, it suffices to show that the right hand side of Inequality~\ref{ineq:sixthpower} is less than or equal to $1$. Let $f(x) =  (2x+1)(x^{5}-7x^{4}+17x^{3}-13x^{2}-5x+8)$ and $g(x) = 2(x-2)^{3}x^{3}$. Note that $f(x) - g(x) = -x^{5}+3x^{4}+7x^{3}-23x^{2}+11x+8$, which gives $f(4)-g(4) < 0$. Moreover, $f'(x)-g'(x) = -5x^4 + 12x^3 + 12x^2 - 46x + 11$ has real roots which are all strictly less than $4$, implying $f'(x) - g'(x) < 0$ when $x\geq 4$. Thus, $f(m) < g(m)$; so, the right hand side of Inequality~\ref{ineq:sixthpower} is indeed less than $1$, implying that Inequality~\ref{ineq:ub} holds.
\end{proof}

We need one last technical lemma before proving Theorem~\ref{thm:upperboundlinear}.

\begin{lem}\label{lemma:rolle}
Let $m\in \mathbb{N}$ with $m\geq 4$. Let $f(x) = (m+1/2)\ln{(m(1-x))}$, \\ $g(x) = (-m^2/2)\ln{(mx / (m-1))}$, and $h(x) = f(x) - g(x)$. If 
\[x_s = \left(1 - \frac{1}{m}\right)e^{2(-c_1m+c_2)/m^2}\text{\qquad and \qquad}x_b = 1-\frac{1}{m},\]
where $c_1 = 1.24$ and $c_2 = 2.05$, then $h(x) > 0$ for all $x_s < x < x_b$.
\end{lem}

Before we prove Lemma~\ref{lemma:rolle}, we first make a brief remark on how we chose the constants $c_1, c_2$. We chose $c_1$ such that it is maximal to two decimal places; i.e., whenever $c_1 \geq 1.25$, the statement of the Lemma is invalid regardless of the choice for $c_2$. Maximizing $c_1$ in this way allows us to obtain the best possible asymptotic bound in Theorem~\ref{thm:upperboundlinear}. We chose $c_2$ so that the statement of the Lemma was still valid for small $m$ and $c_1$ maximal.

\begin{proof}
First note that $h(x_b) = f(x_b) = g(x_b) = 0$. Now, we show that $h(x_s) > 0$. Using the fact that $h(x) = (m+1/2)\ln{(m(1-x))} + (m^2/2)\ln{(mx / (m-1))}$, we have that $h(x_s) > 0$ if and only if the function $p:[4, \infty)\rightarrow \mathbb{R}$ given by
\[p(y) = y - (y-1)e^{2(-c_1y+c_2)/y^2} - e^{c_1}e^{-(c_1+2c_2)/(2y+1)}\]
satisfies $p(m) = e^{f(x_s)/(m+1/2)} - e^{g(x_s)/(m+1/2)} > 0$.

Let $d(x) = e^x - (1+x+x^2/2)$. Clearly, $\lim_{x\rightarrow -\infty}d(x) < 0$. Also, since $d'(x) = e^x - (1+x) > 0$ whenever $x\neq 0$, and $d(0)=0$, we have that $d(x) < 0$ for all $x < 0$. Since $-c_1y + c_2 < 0$ and $-(c_1+2c_2)/(2y+1) < 0$ when $y\geq 4$, the fact that $e^x < 1+x+x^2/2$ for all $x < 0$ implies
\begin{align*}
    p(y) &> \frac{1}{y^4(2y+1)^2}\left(A_{1}\left(2y+1\right)^{2}+A_{2}y\left(2y+1\right)^{2}+A_{3}y^{2}\left(2y+1\right)^{2}+\right. \\
    &\qquad \left.A_{4}y^{3}\left(2y+1\right)^{2}+A_{5}y^{4}\left(2y+1\right)^{2}+A_{6}\left(2y+1\right)y^{4}+A_{7}y^{4}\right),
\end{align*}
where $A_1 = 2c_2^2$, $A_2 = -2c_2^2 - 4c_1c_2$, $A_3 = 4c_1c_2 + 2c_2 + 2c_1^2$, $A_4 = -2c_2-2c_1^2-2c_1$, $A_5 = 2c_1+1-e^{c_1}$, $A_6 = e^{c_{1}}(2c_{2}+c_{1})$, and $A_7 = -e^{c_{1}}(2c_{2}^{2}+2c_{1}c_{2}+c_{1}^{2}/2)$. Let $q(y)$ be the function given by the product of $y^4(2y+1)^2$ and this lower bound on $p(y)$.

Note that $q(y)$ is a degree-six polynomial with leading coefficient $4A_5 > 0$.\footnote{$q(y)\approx 0.0976 y^6 - 1.6164 y^5 - 0.0406 y^4 - 14.563 y^3 - 23.317 y^2 + 15.05 y + 8.405$} If we let $\epsilon = 0.1$, $y_1 = 0.75$, and $y_2 = 17.14$, it is easy to computationally verify $q(y_i+\epsilon)q(y_i-\epsilon)<0$ for each $i\in [2]$, implying by the Intermediate Value Theorem that the two intervals $(y_i-\epsilon, y_i+\epsilon)$ for each $i\in [2]$ contains a real root for $q(y)$. Moreover, Descartes' Rule of Signs implies that these intervals contain the only positive real roots for $q(y)$. Therefore, for all $y\geq 18$, we have $y^4(2y+1)^2p(y) > q(y) > 0$, which further implies that $p(y) > 0$. 
It is easy to verify that $p(k) > 0$ for each integer $k = 4,5,\hdots, 17$. Thus, $h(x_s) > 0$.

Finally, because $\lim_{x\rightarrow 0^+}f(x) = (m+1/2)\ln{m}$, and $\lim_{x\rightarrow 0^+}g(x) = \infty$, we have that $\lim_{x\rightarrow 0+}h(x) < 0$, which implies that $h(x)$ has at least one zero in the interval $(0, x_s)$; we now prove that this zero is unique. Note that this would imply $h(x) > 0$ for all $x_s < x < x_b$, since we have already shown $h(x_b) = 0$ and $h(x_s) > 0$. Assume for the sake of contradiction that $0 < x_r < x_t < x_s$ and $h(x_r) = h(x_t) = 0$. Noting that
\[h'(x) = \frac{1}{2}\left(\frac{2m+1}{x-1} + \frac{m^2}{x}\right),\]
we find that when $x\in (0,1)$, $h'(x) = 0$ if and only if $x = (m / (m+1))^2$. So, by applying Rolle's Theorem to $h(x)$ at the endpoints $x_t$ and $x_b$ (we know $x_t < x_s < x_b$), we have that $x_t < (m / (m+1))^2$. But Rolle's Theorem on $h(x)$ with endpoints $x_r$ and $x_t$ further implies that $x_t > (m / (m+1))^2$, which is a contradiction.
\end{proof}

We are now ready to prove Theorem~\ref{thm:upperboundlinear}.

\begin{proof}
Let $m$ be an arbitrary fixed integer satisfying $n\leq 1.24m - 2.05$. Since \\ $m\geq \left\lceil (n+2.05)/1.24\right\rceil\geq 4$, it suffices to show that $P_{\ell}(G,m) = P(G,m)$.

Take $f(x)$, $g(x)$, $x_s$, and $x_b$ as they are defined in the statement of Lemma~\ref{lemma:rolle}. 
Note that $g(x)$ is continuous on the interval $(x_s, x_b)$. Also, $g(x_b) = 0$ and $g(x_s) = 1.24m-2.05$. Therefore, by the Intermediate Value Theorem, there exists $\epsilon$ satisfying $\epsilon \in (x_s, x_b)$ and $g(\epsilon) = n$. By Lemma~\ref{lemma:rolle}, and the fact that $\epsilon \in (x_s, x_b)$, $g(\epsilon) = n < f(\epsilon)$. Thus, Corollary~\ref{corollary:oneamgmbounds} implies $P_{\ell}(G,m) = P(G,m)$.
\end{proof}

Deducing that $K_{2,3}$ is enumeratively chromatic-choosable and $\tau(K_{2,4}) = \tau(K_{2,5}) = 3$ requires one further result.

\begin{pro}\label{prop:K2n2}
Let $G = K_{2,n}$. If $n=3$, then $P_{\ell}(G, 2) = P(G,2)=2$. Otherwise, if $n\geq 4$, then $0 = P_{\ell}(G,2) < P(G,2) = 2$.
\end{pro}

\begin{proof}
We have $P(G,2) = 2$ for all $n\in \mathbb{N}$. Also, since $\chi_{\ell}(K_{2,n})=3$ whenever $n \geq 4$, it suffices to show that $P_{\ell}(K_{2,3},2)=2$.  For the remainder of the proof, assume $n=3$.

We will prove this by showing that $P(G,L)\geq 2$ over all $2$-assignments $L$ of $G$. Let $G$ have partite sets $X = \{x_1, x_2\}$ and $Y = \{y_1, y_2, y_3\}$. Since $\chi_{\ell}(K_{2,3})=2$, $P(G,L)>0$. 

Now, assume for the sake of contradiction that there exists a $2$-assignment $L'$ such that $P(G,L')=1$. Let $c$ be the only proper $L'$-coloring of $G$, and suppose that $c(x_i) = a_i$ for each $i\in [2]$ and $c(y_j) = b_j$ for each $j\in [3]$. Then, we must have that $L'(x_i) = \{a_i, c_i\}$ and $c_i\in \{b_1, b_2, b_3\}$ for each $i\in [2]$. If not, then there exists $x\in X$ such that coloring $x$ with the element in $L'(x)$ not equal to $c(x)$ yields another proper $L'$-coloring, which contradicts our original assumption. Using a similar argument, we can conclude $L'(y_i) = \{b_i, d_i\}$ and $d_i\in \{a_1, a_2\}$ for each $i\in [3]$. But then, the $L'$-coloring $c'$ satisfying $c'(v)\neq c(v)$ for all $v\in V(G)$ is proper, contradicting $P(G,L')=1$.
\end{proof}

\begin{cor}
We have $\tau(K_{2,3}) = 2$ and $\tau(K_{2,4}) = \tau(K_{2,5}) = 3$. 
\end{cor}

\begin{proof}
By Theorem~\ref{thm:upperboundlinear}, $\tau(K_{2,3})\leq 4$, implying $P_{\ell}(K_{2,3},m) = P(K_{2,3},m)$ for all $m\geq 4$. Theorem~\ref{thm: casework} give that
$P_{\ell}(K_{2,3},m) = P(K_{2,3}, m)$ for each $m\in \{2,3\}$. Since $\chi_{\ell}(K_{2,3}) = 2$, it follows that $\tau(K_{2,3}) = 2$.

Also, Theorem~\ref{thm:upperboundlinear} gives $\tau(K_{2,4})\leq 5$, implying $P_{\ell}(K_{2,4},m) = P(K_{2,4},m)$ for all $m\geq 5$. Theorem~\ref{thm: casework} gives that $P_{\ell}(K_{2,4},m) = P(K_{2,4}, m)$ for each $m\in \{3,4\}$. Since $\chi_{\ell}(K_{2,4}) = 3$, it follows that $\tau(K_{2,4}) = 3$.

Finally, Theorem~\ref{thm:upperboundlinear} also gives $\tau(K_{2,5})\leq 6$, implying $P_{\ell}(K_{2,5},m) = P(K_{2,5},m)$ for all $m\geq 6$. Theorem~\ref{thm: casework} gives that $P_{\ell}(K_{2,5},m) = P(K_{2,5}, m)$ for each $m\in \{3,4,5\}$. Since $\chi_{\ell}(K_{2,5}) = 3$, it follows that $\tau(K_{2,5}) = 3$.
\end{proof}

{\bf Acknowledgment.}  This paper is based on a research project conducted with undergraduate students Akash Kumar, Andrew Liu, Patrick Rewers, Paul Shin, Michael Tanahara, and Khue To at the College of Lake County during the summer 2021, fall 2021, and spring 2022 semesters. The support of the College of Lake County is gratefully acknowledged.  The authors would also like to thank Doug West for helpful conversations.

\appendix
\newpage
\section{Appendix} \label{finish}

In this Appendix we prove Statements~(ii) and~(iii) of Theorem~\ref{thm: casework}.

\begin{proof}
Let $G = K_{2,n}$ with bipartition $\{x_1, x_2\}$, $\{y_1, \hdots, y_n\}$.  Since $\tau(C_4) = 2$ and we wish to begin by proving the first part of Statement~(ii), we will assume that $3 \leq n \leq 24$.  Let and $L$ be a $4$-assignment for $G$ such that $L(y_j)\subseteq L(x_1)\cup L(x_2)$ for all $j\in [n]$ and $P(G,L) = P_{\ell}(G,4)$ (we know such an $L$ exists by Lemma~\ref{lemma:unionbadgeneral}). For each $(i,j)\in L(x_1)\times L(x_2)$ and $k\in [n]$, let $q_{k, (i,j)} = \vert L(y_k) - \{i,j\}\vert$. Let $B = L(x_1) - L(x_2)$, $C = L(x_2) - L(x_1)$, $D = L(x_1)\cap L(x_2)$ and $d = \vert D\vert$. Clearly, $d\in \{0,1,2,3,4\}$. We will show that $P(G,L)\geq P(G,4) = 4\cdot 3^n + 12\cdot 2^n$ for each possible $d$. From this, it follows that $P_{\ell}(G,4) = P(G,4)$.

If $d=4$, then clearly $P(G,L) = P(G,4)$. If $d=3$, then by Lemma~\ref{lemma:dm-1bad} we have $P(G,L)\geq P(G,4)$. If $d=2$, then by Lemma~\ref{lemma:dm-2} we have 
\begin{align*}
    P(G,L) &\geq 2(9^{a_1}9^{a_2}12^{a_3}16^{a_4})^{1/2} + 2(4^{a_4}3^{a_3}2^{a_1+a_2})^{1/2} \\
    &\qquad + 8\cdot 3^{n/2}(16^{a_1}16^{a_2}32^{a_3}81^{a_4})^{1/8} + 4(81^{a_1}72^{a_2}36^{a_3}16^{a_4})^{1/4}
\end{align*}
for all ordered quadruples of nonnegative integers $(a_1, a_2, a_3, a_4)$ satisfying $a_1+a_2+a_3+a_4=n$. It is easy to verify computationally that the expression above is greater than or equal to $P(G,4)$ over all such ordered quadruples, for all $3\leq n\leq 24$. See the code for the first program in Appendix~\ref{code}.

So, suppose $d=1$. Without loss of generality, assume $L(x_1) = \{1,2,3,4\}$ and $L(x_2) = \{1,5,6,7\}$. Then $L(y_k)\in \binom{[7]}{4}$ for each $k\in [n]$. For each $A\in \binom{[7]}{4}$, let $z_{A} = \vert \{k\in [n]: L(y_k) = A\}\vert$. Note that $\sum_{A\in \binom{[7]}{4}}z_A = n$. 

Now, let $\{\mathcal{X}, \mathcal{Y}, \mathcal{Z}, \mathcal{W}\}$ be the partition of $\binom{[7]}{4}$ satisfying 
\begin{align*}
    \mathcal{X} &= \{X: (D \subset X)\} \cap (\{X: \vert X\cap B\vert = 2\}\cup \{X: \vert X\cap B\vert = 0\}), \\
    \mathcal{Y} &= \{Y: (D \subset Y)\}\cap \{Y: \vert Y\cap B\vert = 1\}, \\
    \mathcal{Z} &= \{Z: \vert Z\cap D\vert = 1\}\cap (\{Z: \vert Z\cap B\vert = 2\}\cup \{Z: \vert Z\cap B\vert = 1\}),\text{ and}\\
    \mathcal{W} &= \{W: \vert W\cap D\vert = 0\}.
\end{align*}

Broadly speaking, this partition serves to separate the set of possible lists for each $y_k$ into four distinct list types. Let $a_1 = \sum_{S\in \mathcal{X}}z_S$, $a_2 = \sum_{S\in \mathcal{Y}}z_S$, $a_3 = \sum_{S\in \mathcal{Z}}z_S$, and $a_4 = \sum_{S\in \mathcal{W}}z_S$.

Now, we compute lower bounds for the values of $\vert \mathcal{C}_{(i,j)}\vert$ across all $(i,j)\in L(x_1)\times L(x_2)$. More specifically, for each $(i,j)\in L(x_1)\times L(x_2)$, we analyze the value of $q_{k,(i,j)}$ for each $y_k$ by performing casework on each of the four possible list types as described in the previous paragraph. Note that we necessarily have $q_{k,(i,j)}\in \{2,3,4\}$.

Suppose $(i,j) = (1,1)$. In this case, $\vert L(y_k)\cap \{1\}\vert = 1$ if and only if $L(y_k)\in \mathcal{X}$, or $L(y_k)\in \mathcal{Y}$; in all other cases, $\vert L(y_k)\cap \{1\}\vert = 0$. Thus, $\vert \{y_k: q_{k,(i,j)} = 3\}\vert = a_1+a_2$ and $\vert \{y_k: q_{k,(i,j)} = 4\}\vert = a_3+a_4$. So, we have that
\begin{equation}\label{eq:m=4dm-31}
\vert\mathcal{C}_{(1,1)}\vert = 4^{a_3 + a_4}3^{a_1+a_2}.
\end{equation}

Next, consider all $(i,j)\in H$ where $H = ([1] \times C) \cup (B \times [1])$.  When $L(y_k)\in \mathcal{X}$, then we have $\vert L(y_k)\cap B\vert = 3$ and $\vert L(y_k)\cap C\vert = 0$, or $\vert L(y_k)\cap B\vert = 0$ and $\vert L(y_k)\cap C\vert = 3$. Applying Statements 3 and 4 in Lemma~\ref{lemma:counting} then gives $\vert \{(i,j)\in H: q_{k,(i,j)} = 3\}\vert = 3$ and $\vert \{(i,j)\in H: q_{k, (i,j)} = 2\}\vert = 3$. Then, by repeatedly applying Statements 3 and 4 in Lemma~\ref{lemma:counting} in the same manner for each of the three remaining list types, we get that: if $L(y_k) \in \mathcal{Y}$, then $\vert \{(i,j)\in H: q_{k, (i,j)}=2\}\vert = 3$ and $\vert \{(i,j)\in H: q_{k, (i,j)}=3\}\vert = 3$; and, if $L(y_k) \in \mathcal{Z}$, or $L(y_k)\in \mathcal{W}$, then $\vert \{(i,j)\in H: q_{k, (i,j)}=3\}\vert = 4$ and $\vert \{(i,j)\in H: q_{k, (i,j)}=4\}\vert = 2$. In sum, our casework implies that $\sum_{(i,j)\in H} \vert \{y_k: q_{k,(i,j)} = 2\}\vert = 3a_1+3a_2$, $\sum_{(i,j)\in H} \vert \{y_k: q_{k,(i,j)} = 3\}\vert = 3a_1+3a_2+4a_3+4a_4$, and $\sum_{(i,j)\in H} \vert \{y_k: q_{k,(i,j)} = 4\}\vert = 2a_3+2a_4$. So, by the AM-GM inequality:
\begin{equation}\label{eq:m=4dm-32}
\sum_{(i,j)\in H}\vert \mathcal{C}_{(i,j)}\vert \geq 6(4^{2a_3 + 2a_4}3^{3a_1+3a_2+4a_3+4a_4}2^{3a_1+3a_2})^{1/6}.
\end{equation}

Finally, consider all $(i,j)\in J$ where $J = B \times C$. When $L(y_k)\in \mathcal{X}$, then Statement 5 in Lemma~\ref{lemma:counting} gives $\vert \{(i,j)\in J: q_{k,(i,j)} = 3\}\vert = 9$. By repeatedly applying Statement 5 in the same manner for each of the three remaining list types, we get that: if $L(y_k)\in \mathcal{Y}$, then $\vert \{(i,j)\in J: q_{k,(i,j)} = 4\}\vert = 2$, $\vert \{(i,j)\in J: q_{k,(i,j)} = 3\}\vert = 5$, and $\vert \{(i,j)\in J: q_{k,(i,j)} = 2\}\vert = 2$; if $L(y_k)\in \mathcal{Z}$, then $\vert \{(i,j)\in J: q_{k,(i,j)} = 3\}\vert = 6$, $\vert \{(i,j)\in J: q_{k,(i,j)} = 2\}\vert = 3$; and, if $L(y_k)\in \mathcal{W}$, then $\vert \{(i,j)\in J: q_{k,(i,j)} = 4\}\vert = 1$, $\vert \{(i,j)\in J: q_{k,(i,j)} = 3\}\vert = 4$, and $\vert \{(i,j)\in J: q_{k,(i,j)} = 2\}\vert = 4$. In sum, our casework implies that $\sum_{(i,j)\in J} \vert \{y_k: q_{k, (i,j)} = 2\}\vert = 2a_2+3a_3+4a_4$, $\sum_{(i,j)\in J} \vert \{y_k: q_{k, (i,j)} = 3\}\vert = 9a_1+5a_2+6a_3+4a_4$, and $\sum_{(i,j)\in J} \vert \{y_k: q_{k, (i,j)} = 4\}\vert = 2a_2+a_4$. So, by the AM-GM inequality:

\begin{equation}\label{eq:m=4dm-33}
\sum_{(i,j)\in J}\vert \mathcal{C}_{(i,j)}\vert \geq 9(4^{2a_2+a_4}3^{9a_1+5a_2+6a_3+4a_4}2^{2a_2+3a_3+4a_4})^{1/9}.
\end{equation}

Finally, Equation~\ref{eq:m=4dm-31} and Inequalities~\ref{eq:m=4dm-32} and \ref{eq:m=4dm-33} give us that
\begin{align*}
    P(G,L) &= \sum_{(i,j)\in L(x_1)\times L(x_2)} \vert\mathcal{C}_{(i,j)}\vert \\
    &=  \vert\mathcal{C}_{(1,1)}\vert + \sum_{(i,j)\in H} \vert\mathcal{C}_{(i,j)}\vert + \sum_{(i,j)\in J} \vert\mathcal{C}_{(i,j)}\vert \\
    &\geq 3^{a_1}3^{a_2}4^{a_3}4^{a_4}+6\cdot 6^{a_1/2}6^{a_2/2}36^{a_3/3}36^{a_4/3}+9\cdot 3^{a_1}15552^{a_2/9}18^{a_3/3}5184^{a_4/9} \\
    &\geq 3^{x}4^{n-x}+6\cdot 6^{x/2}36^{(n-x)/3}+9\cdot 15552^{x/9}5184^{(n-x)/9},
\end{align*}
where $x = a_1+a_2$. It is easy to verify computationally that the expression above is at least $P(G,4)$ over all $0\leq x\leq n$, for all $3\leq n\leq 24$. See the code for the second program in Appendix~\ref{code}.

Finally, suppose $d=0$. Without loss of generality, assume $L(x_1) = \{1,2,3,4\}$ and $L(x_2) = \{5,6,7,8\}$. Then $L(y_k)\in \binom{[8]}{4}$ for each $k\in [n]$. For each $A\in \binom{[8]}{4}$, let $z_{A} = \vert \{k\in [n]: L(y_k) = A\}\vert$. Notice that $\sum_{A\in \binom{[8]}{4}}z_A = n$. Now, let $\{\mathcal{X}, \mathcal{Y}, \mathcal{Z}\}$ be the partition of $\binom{[8]}{4}$ satisfying
\begin{align*}
    \mathcal{X} &= \{X: \vert X\cap B\vert = 4\}\cup \{X: \vert X\cap B\vert = 0\}, \\
    \mathcal{Y} &= \{Y: \vert Y\cap B\vert = 3\}\cup \{Y: \vert Y\cap B\vert = 1\}, \text{ and}\\
    \mathcal{Z} &= \{Z: \vert Z\cap B\vert = 2\}.
\end{align*}
As before, this partition serves to separate the set of possible lists for each $y_k$ into three distinct list types. Let $a_1 = \sum_{S\in \mathcal{X}}z_S$, $a_2 = \sum_{S\in \mathcal{Y}}z_S$, and $a_3 = \sum_{S\in \mathcal{Z}}z_S$. 

Now, using the same general strategy that we used in the case when $d=1$, we compute lower bounds for the values of $\vert \mathcal{C}_{(i,j)}\vert$ across all $(i,j)\in L(x_1)\times L(x_2)$. When $L(y_k)\in \mathcal{X}$, then we have $\vert L(y_k)\cap L(x_1)\vert = 4$ or $\vert L(y_k)\cap L(x_1)\vert = 0$. For such $k$, applying all five statements in Lemma~\ref{lemma:counting} gives $\vert \{(i,j)\in L(x_1)\times L(x_2): q_{k, (i,j)} = 3\}\vert = 16$. Applying Lemma~\ref{lemma:counting} in the same manner for each of the two remaining list types, we get that: if $L(y_k)\in \mathcal{Y}$, then $\vert \{(i,j)\in L(x_1)\times L(x_2): q_{k, (i,j)}=4\}\vert = 3$, $\vert \{(i,j)\in L(x_1)\times L(x_2): q_{k, (i,j)} = 3\}\vert = 10$, and $\vert \{(i,j)\in L(x_1)\times L(x_2): q_{k, (i,j)} = 2\}\vert = 3$; and, if $L(y_k)\in \mathcal{Z}$, then $\vert \{(i,j)\in L(x_1)\times L(x_2): q_{k, (i,j)} = 4\}\vert = 4$, $\vert \{(i,j)\in L(x_1)\times L(x_2): q_{k, (i,j)} = 3\}\vert = 8$, and $\vert \{(i,j)\in L(x_1)\times L(x_2): q_{k, (i,j)} = 2\}\vert = 4$. In sum, our casework implies that $\sum_{(i,j)\in L(x_1)\times L(x_2)} \vert \{y_k: q_{k, (i,j)} = 2\}\vert = 3a_2+4a_3$, $\sum_{(i,j)\in L(x_1)\times L(x_2)} \vert \{y_k: q_{k, (i,j)} = 3\}\vert = 10a_2+8a_3+16a_1$, and $\sum_{(i,j)\in L(x_1)\times L(x_2)} \vert \{y_k: q_{k, (i,j)} = 4\}\vert = 3a_2+4a_3$. So, by the AM-GM inequality:
\begin{align*}
    P(G,L) &= \sum_{(i,j)\in L(x_1)\times L(x_2)}\vert \mathcal{C}_{(i,j)}\vert\\
    &\geq 16(2^{3a_2+4a_3}3^{10a_2+8a_3+16a_1}4^{3a_2+4a_3})^{1/16} \\
    &= 16\cdot 3^{n/2}(6561^{a_1}4608^{a_2}4096^{a_3})^{1/16}\\
    &\geq 16\cdot 3^{n/2}4096^{n/16}.
\end{align*}
It is easy to verify computationally that the expression above is greater than or equal to $P(G,4)$ for all $3\leq n\leq 24$. See the code for the second program in Appendix~\ref{code}.
Since we have exhausted all possible values for $d$, the proof of the first part of Statement~(ii) is complete.

Now, we turn our attention to the second part of Statement~(ii).  If $n \geq 32$, Theorem~\ref{theorem:firstthresholdpaper} implies $P_{\ell}(K_{2,n},4) < P(K_{2,n},4)$ since $\lfloor n/4\rfloor\geq 9\ln(16/17)$. On the other hand, if $n \in \{27, 28, 29, 30, 31\}$, it is easy to verify from Lemma~\ref{lemma:balancedextension} that there is a 4-assignment $L$ for $G$ with the property that $P(G,L) < P(G,4)$.

Now, we turn our attention to Statement~(iii).  Let $G = K_{2,n}$ with bipartition $\{x_1, x_2\}$, $\{y_1, \hdots, y_n\}$.  Since $\tau(C_4) = 2$ and we wish to begin by proving the first part of Statement~(iii), we will assume that $3 \leq n \leq 43$.  Let $L$ be a $5$-assignment for $G$ such that $L(y_j)\subseteq L(x_1)\cup L(x_2)$ for all $j\in [n]$ and $P(G,L) = P_{\ell}(G,5)$ (we know such an $L$ exists by Lemma~\ref{lemma:unionbadgeneral}). For each $(i,j)\in L(x_1)\times L(x_2)$ and $k\in [n]$, let $q_{k, (i,j)} = \vert L(y_k) - \{i,j\}\vert$. Let $B = L(x_1) - L(x_2)$, $C = L(x_2) - L(x_1)$, $D = L(x_1)\cap L(x_2)$ and $d = \vert D\vert$. Clearly, $d\in \{0,1,2,3,4,5\}$. We will show that $P(G,L)\geq P(G,5) = 5\cdot 4^n + 20\cdot 3^n$ for each possible $d$. From this, it follows that $P_{\ell}(G,5) = P(G,5)$.

If $d=5$, then clearly $P(G,L) = P(G,5)$. If $d=4$, then by Lemma~\ref{lemma:dm-1bad} we have $P(G,L)\geq P(G,5)$. And, if $d=3$, then by Lemma~\ref{lemma:dm-2} we have 

\begin{align*}
    P(G,L)&\geq 3\cdot 4^{a_1}4^{a_2}80^{a_3/3}100^{a_4/3}+6\cdot 3^{a_1}3^{a_2}48^{a_3/3}80^{a_4/3}\\
    &\qquad +12\cdot 12^{a_1/2}12^{a_2/2}3732480^{a_3/12}48^{a_4/3}+4\cdot 4^{a_1}240^{a_2/4}12^{a_3/2}3^{a_4}
\end{align*}
for all ordered quadruples of nonnegative integers $(a_1, a_2, a_3, a_4)$ satisfying $a_1+a_2+a_3+a_4=n$. It is easy to verify computationally that the expression above is greater than or equal to $P(G,5)$ over all such ordered quadruples, for all $3\leq n\leq 43$. See the code for the first program in Appendix~\ref{code}. 

So, suppose $d=2$. Without loss of generality, assume $L(x_1) = \{1,2,3,4,5\}$ and $L(x_2) = \{1,2,6,7,8\}$. Then $L(y_k)\in \binom{[8]}{5}$ for each $k\in [n]$. For each $A\in \binom{[8]}{5}$, let $z_{A} = \vert \{k\in [n]: L(y_k) = A\}\vert$. Note that $\sum_{A\in \binom{[8]}{5}}z_A = n$. 

Now, let $\{\mathcal{X}, \mathcal{Y}, \mathcal{Z}, \mathcal{W}, \mathcal{V}\}$ be the partition of $\binom{[8]}{5}$ satisfying 
\begin{align*}
    \mathcal{X} &= \{X: (D \subset X)\} \cap (\{X: \vert X\cap B\vert = 3\}\cup \{X: \vert X\cap B\vert = 0\}), \\
    \mathcal{Y} &= \{Y: (D \subset Y)\}\cap (\{Y: \vert Y\cap B\vert = 2\}\cup \{Y: \vert Y\cap B\vert = 1\}), \\
    \mathcal{Z} &= \{Z: \vert Z\cap D\vert = 1\}\cap (\{Z: \vert Z\cap B\vert = 3\}\cup \{Z: \vert Z\cap B\vert = 1\}),\\
    \mathcal{W} &= \{W: \vert W\cap D\vert = 1\} \cap \{W: \vert W\cap B\vert = 2\},\text{ and} \\
    \mathcal{V} &= \{V: \vert V\cap D\vert = 0\} \cap (\{V: \vert V\cap B\vert = 3\}\cup \{V: \vert V\cap B\vert = 2\}).
\end{align*}
Broadly speaking, this partition serves to separate the set of possible lists for each $y_k$ into five distinct list types. Let $a_1 = \sum_{S\in \mathcal{X}}z_S$, $a_2 = \sum_{S\in \mathcal{Y}}z_S$, $a_3 = \sum_{S\in \mathcal{Z}}z_S$, $a_4 = \sum_{S\in \mathcal{W}}z_S$, and $a_5 = \sum_{S\in \mathcal{V}}z_S$.

Now, we compute lower bounds for the values of $\vert \mathcal{C}_{(i,j)}\vert$ across all $(i,j)\in L(x_1)\times L(x_2)$. More specifically, for each $(i,j)\in L(x_1)\times L(x_2)$, we analyze the value of $q_{k, (i,j)}$ for each $y_k$ by performing casework on each of the five possible list types as described in the previous paragraph. Note that we necessarily have $q_{k, (i,j)} \in \{3,4,5\}$.

Consider when $(i,j)\in E$ where $E = \{(1,1),(2,2) \}$. When $L(y_k) \in \mathcal{X}$, then we have $\vert L(y_k)\cap B\vert = 2$ and $\vert L(y_k)\cap C\vert = 0$. Applying Statement 1 in Lemma~\ref{lemma:counting}, we find that $\vert \{(i,j)\in E: q_{k, (i,j)} = 4\}\vert = 2$. By repeatedly applying Statement 1 in Lemma~\ref{lemma:counting} in the same way for each of the four remaining list types, we get that: if $L(y_k)\in \mathcal{Y}$, then $\vert \{(i,j)\in E: q_{k, (i,j)} = 4\}\vert = 2$; if $L(y_k)\in \mathcal{Z}$ or $L(y_k)\in \mathcal{W}$, then $\vert \{(i,j)\in E: q_{k, (i,j)} = 4\}\vert = 1$ and $\vert \{(i,j)\in E: q_{k, (i,j)} = 5\}\vert = 1$; and, if $L(y_k)\in \mathcal{V}$, then $\vert \{(i,j)\in E: q_{k, (i,j)} = 5\}\vert = 2$. In sum, our casework implies that $\sum_{(i,j)\in E} \vert \{y_k: q_{k, (i,j)} = 4\}\vert = 2a_1+2a_2+a_3+a_4$, and $\sum_{(i,j)\in E} \vert \{y_k: q_{k, (i,j)} = 5\}\vert = a_3+a_4+2a_5$. So, by the AM-GM inequality:
\begin{equation}\label{eq:m=5dm-3amgm1}
\sum_{(i,j)\in E}\vert\mathcal{C}_{(i,j)}\vert \geq 2\cdot 4^{a_1}4^{a_2}20^{a_3/2}20^{a_4/2}5^{a_5}.
\end{equation}

Let $F = \{(a,b): a,b\in [2], a \neq b\}$, $H = ([2] \times C) \cup (B \times [2])$, and $J = B \times C$. By considering each of the three cases when $(i,j)\in F,H,J$ in the same manner as the previous paragraph (using all five statements in Lemma~\ref{lemma:counting}) we obtain the following inequalities by the AM-GM inequality : 
\begin{align*}
    \sum_{(i,j)\in F}\vert\mathcal{C}_{(i,j)}\vert &\geq 2\cdot 4^{a_1}4^{a_2}20^{a_3/2}20^{a_4/2}5^{a_5}, \\
    \sum_{(i,j)\in H}\vert\mathcal{C}_{(i,j)}\vert &\geq 12\cdot 12^{a_1/2}12^{a_2/2}2880^{a_3/6}2880^{a_4/6}5120^{a_5/6},\text{ and} \\
    \sum_{(i,j)\in J}\vert\mathcal{C}_{(i,j)}\vert &\geq 9\cdot 4^{a_1}230400^{a_2/9}48^{a_3/3}103680^{a_4/9}36^{a_5/3}.
\end{align*}
Then, we have 
\begin{align*}
    P(G,L) &= \sum_{(i,j)\in L(x_1)\times L(x_2)} \vert\mathcal{C}_{(i,j)}\vert \\
    &= \sum_{(i,j)\in E} \vert\mathcal{C}_{(i,j)}\vert + \sum_{(i,j)\in F} \vert\mathcal{C}_{(i,j)}\vert + \sum_{(i,j)\in H} \vert\mathcal{C}_{(i,j)}\vert + \sum_{(i,j)\in J} \vert\mathcal{C}_{(i,j)}\vert \\
    &\geq 4\cdot 4^{a_1}4^{a_2}20^{a_3/2}20^{a_4/2}5^{a_5} + 12\cdot 12^{a_1/2}12^{a_2/2}2880^{a_3/6}2880^{a_4/6}5120^{a_5/6}\\
    &\qquad + 9\cdot 4^{a_1}230400^{a_2/9}48^{a_3/3}103680^{a_4/9}36^{a_5/3}.
\end{align*}
It is easy to verify computationally that the expression above is greater than or equal to $P(G,5)$ over all ordered quintuples of nonnegative integers $(a_1, a_2, a_3, a_4, a_5)$ satisfying $a_1+a_2+a_3+a_4+a_5=n$ for all $3\leq n\leq 43$. See the code for the third program in Appendix~\ref{code}.

So, suppose $d=1$. Without loss of generality, assume $L(x_1) = \{1,2,3,4,5\}$ and $L(x_2) = \{1,6,7,8,9\}$. Then $L(y_k)\in \binom{[9]}{5}$ for each $k\in [n]$. For each $A\in \binom{[9]}{5}$, let $z_{A} = \vert \{k\in [n]: L(y_k) = A\}\vert$. Note that $\sum_{A\in \binom{[9]}{5}}z_A = n$. 

Now, let $\{\mathcal{X}, \mathcal{Y}, \mathcal{Z}, \mathcal{W}, \mathcal{V}\}$ be the partition of $\binom{[9]}{5}$ satisfying 
\begin{align*}
    \mathcal{X} &= \{X: (1\in X)\} \cap (\{X: \vert X\cap B\vert = 4\}\cup \{X: \vert X\cap B\vert = 0\}), \\
    \mathcal{Y} &= \{Y: (1\in Y)\}\cap (\{Y: \vert Y\cap B\vert = 3\}\cup \{Y: \vert Y\cap B\vert = 1\}), \\
    \mathcal{Z} &= \{Z: (1\in Z)\}\cap\{Z: \vert Z\cap B\vert = 2\}, \\
    \mathcal{W} &= \{W: (1\notin W)\} \cap (\{W: \vert W\cap B\vert = 4\}\cup \{W: \vert W\cap B\vert = 1\}),\text{ and} \\
    \mathcal{V} &= \{V: (1\notin V)\} \cap (\{V: \vert W\cap B\vert = 3\}\cup \{W: \vert W\cap B\vert = 2\}).
\end{align*}
As before, this partition serves to separate the set of possible lists for each $y_k$ into five distinct list types. Let $a_1 = \sum_{S\in \mathcal{X}}z_S$, $a_2 = \sum_{S\in \mathcal{Y}}z_S$, $a_3 = \sum_{S\in \mathcal{Z}}z_S$, $a_4 = \sum_{S\in \mathcal{W}}z_S$, and $a_5 = \sum_{S\in \mathcal{V}}z_S$.

Now, using the same general strategy that we used in the case when $d=2$, we compute lower bounds for the values of $\vert \mathcal{C}_{(i,j)}\vert$ across all $(i,j)\in L(x_1)\times L(x_2)$. Let $H = \{(1,b): b\in C\}\cup (\{(a,1): a\in B\}$ and $J = B \times C$. We can use Lemma~\ref{lemma:counting} and the AM-GM inequality to obtain:
\begin{align*}
    \vert\mathcal{C}_{(1,1)}\vert &= 4^{a_1}4^{a_2}4^{a_3}5^{a_4}5^{a_5}, \\
    \sum_{(i,j)\in H}\vert\mathcal{C}_{(i,j)}\vert &\geq 8\cdot 12^{a_1/2}12^{a_2/2}12^{a_3/2}128000^{a_4/8}128000^{a_5/8},\text{ and} \\
    \sum_{(i,j)\in J}\vert\mathcal{C}_{(i,j)}\vert &\geq 16\cdot 4^{a_1}3538944000^{a_2/16}240^{a_3/4}192^{a_4/4}34560^{a_5/8}.
\end{align*}
Then, we have 
\begin{align*}
    P(G,L) &= \sum_{(i,j)\in L(x_1)\times L(x_2)} \vert\mathcal{C}_{(i,j)}\vert \\
    &=  \vert\mathcal{C}_{(1,1)}\vert + \sum_{(i,j)\in H} \vert\mathcal{C}_{(i,j)}\vert + \sum_{(i,j)\in J} \vert\mathcal{C}_{(i,j)}\vert \\
    &\geq 4^{a_1}4^{a_2}4^{a_3}5^{a_4}5^{a_5} + 8\cdot 12^{a_1/2}12^{a_2/2}12^{a_3/2}128000^{a_4/8}128000^{a_5/8}\\
    &\qquad + 16\cdot 4^{a_1}3538944000^{a_2/16}240^{a_3/4}192^{a_4/4}34560^{a_5/8}.
\end{align*}
It is easy to verify computationally that the expression above is greater than or equal to $P(G,5)$ over all ordered quintuples of nonnegative integers $(a_1, a_2, a_3, a_4, a_5)$ satisfying $a_1+a_2+a_3+a_4+a_5 = n$ for all $3\leq n\leq 43$. See the code for the third program in Appendix~\ref{code}. 

Finally, suppose $d=0$. Without loss of generality, assume $L(x_1) = \{1,2,3,4,5\}$ and $L(x_2) = \{6,7,8,9,10\}$. Then $L(y_k)\in \binom{[10]}{5}$ for each $k\in [n]$. For each $A\in \binom{[10]}{5}$, let $z_{A} = \vert \{k\in [n]: L(y_k) = A\}\vert$. Note that $\sum_{A\in \binom{[10]}{5}}z_A = n$. 

Now, let $\{\mathcal{X}, \mathcal{Y}, \mathcal{Z}\}$ be the partition of $\binom{[10]}{4}$ satisfying 
\begin{align*}
    \mathcal{X} &= \{X: \vert X\cap B\vert = 5\}\cup \{X: \vert X\cap B\vert = 0\}, \\
    \mathcal{Y} &= \{Y: \vert Y\cap B\vert = 4\}\cup \{Y: \vert Y\cap B\vert = 1\}, \text{ and}\\
    \mathcal{Z} &= \{Z: \vert Z\cap B\vert = 3\}\cup\{Z: \vert Z\cap B\vert = 2\}.
\end{align*}
As in the previous two cases, this partition serves to separate the set of possible lists for each $y_k$ into three distinct list types. Let $a_1 = \sum_{S\in \mathcal{X}}z_S$, $a_2 = \sum_{S\in \mathcal{Y}}z_S$, and $a_3 = \sum_{S\in \mathcal{Z}}z_S$.

Now, using Lemma~\ref{lemma:counting} and the AM-GM inequality we can compute a lower bound for the values of $\vert \mathcal{C}_{(i,j)}\vert$ across all $(i,j)\in L(x_1)\times L(x_2)$:
\begin{align*}
    P(G,L) = \sum_{(i,j)\in L(x_1)\times L(x_2)}\vert\mathcal{C}_{(i,j)}\vert &\geq 25\cdot 4^{a_1}(3^44^{17}5^{4})^{a_2/25}(3^64^{13}5^6)^{a_3/25}.
\end{align*}
It is easy to verify computationally that the right hand side of the above inequality is greater than or equal to $P(G,5)$ over all ordered triples of nonnegative integers $(a_1, a_2, a_3)$ satisfying $a_1+a_2+a_3 = n$ for all $3\leq n\leq 43$. See the code for the third program in Appendix~\ref{code}.  Since we have now exhausted all possible values for $d$, the proof of the first part of Statement~(iii) is complete. 

Now, we turn our attention to the second part of Statement~(iii).  If $n \geq 56$, Theorem~\ref{theorem:firstthresholdpaper} implies $P_{\ell}(K_{2,n},5) < P(K_{2,n},5)$ since $\lfloor n/4\rfloor\geq 16\ln(16/17)$. On the other hand, if $n \in [55]-[43]$, it is easy to verify from Lemma~\ref{lemma:balancedextension} that there is a 5-assignment $L$ for $G$ with the property that $P(G,L) < P(G,5)$.
\end{proof}

\section{Appendix} \label{code}

All of the programs in this Appendix are written in Python. Our first program allows us to find the $m$ and $n$ values for which the lower bound in Lemma~\ref{lemma:dm-2} demonstrates that $P_{\ell}(K_{2,n},m) < P(K_{2,n},m)$. 

\begin{verbatim}
import math

n=1
m=3 #choose any desired value of $m\geq 3$

def PGL_1(x, y, z, w, m):
    return (m-2) * (
        ((m-1) ** (x + y + z * (m-3) / (m-2) + w * (m-4) / (m-2))) * 
        (m ** ((z + 2*w) / (m-2)))
    )

def PGL_2(x, y, z, w, m):
    if (m==3):
        return 0
    return (m-2)*(m-3) * (
        ((m-2) ** (x + y + z * (m-4) / (m-2) + w * (m-4)*(m-5) / ((m-2) * (m-3)))) *
        ((m-1) ** (z * 2 / (m-2) + w * 4 * (m-4) / ((m-2) * (m-3)))) *
        (m ** (w * 2 / ((m-2) * (m-3))))
    )

def PGL_3(x, y, z, w, m):
    return 4*(m-2)* (
        ((m-2) ** (x / 2 + y / 2 + z * 3*(m-3) / (4 * (m-2)) + w * (m-4) / (m-2))) *
        ((m-1) ** (x / 2 + y / 2 + z * m / (4 * (m-2)) + w * 2 / (m-2))) *
        (m ** (z / (4 * (m-2))))
    )

def PGL_4(x, y, z, w, m):
    return 4 * (
        ((m-2) ** (y / 4 +z / 2 + w)) *
        ((m-1) ** (x + y / 2 + z / 2)) *
        (m ** (y /4))
    )

# This function is the lower bound for P(G,L) given in the statement of the Lemma.
def PGL(x,y,z,w,m):
    return PGL_1(x,y,z,w,m)+PGL_2(x,y,z,w,m)+PGL_3(x,y,z,w,m)+PGL_4(x,y,z,w,m)

# This function is $P(K_{2,n},m)$.
def PG(n,m):
    return m * ((m-1)**(n)) + m*(m-1) * ((m-2)**(n))

stop = False
chromatic_polynomial = 0
# This loop runs for all $m\geq 4$. 
if (m!=3):
    while (n >= 1):
        chromatic_polynomial = PG(n,m)
        for x in range(n+1):
            for y in range(n+1-x):
                for z in range(n+1-x-y):
                    w = n-x-y-z
                    if (chromatic_polynomial > PGL(x,y,z,w,m)):
                        stop = True
        if (not stop):
            print("n = " + str(n) + " is good")
            n += 1
        else:
            print("n = " + str(n) + " is the first bad n")
            break
# This loop runs when m=3. It takes into account the fact that $w$ is 
  necessarily equal to $0$ when $m=3$.
if (m==3):
    while (n >= 1):
        chromatic_polynomial = PG(n,m)
        for x in range(n+1):
            for y in range(n+1-x):
                z = n-x-y
                if (chromatic_polynomial > PGL(x,y,z,0,m)):
                    stop = True
        if (not stop):
            print("n = " + str(n) + " is good")
            n += 1
        else:
            print("n = " + str(n) + " is the first bad n")
            break
\end{verbatim}

Our second program is needed for the proof of Statement~(ii) of Theorem~\ref{thm: casework}.

\begin{verbatim}
import math

n=1
m=4

# This function is the lower bound for $P(G,L)$ when $d=1$.
def PGL_d1(x):
    return ((3**(x) * 4**(n-x)) + 
    6*(6**(x/2) * 36**((n-x)/3)) + 
    9*(15552**(x/9) * 5184**((n-x)/9)))

# This function is the lower bound for $P(G,L)$ when $d=0$.
def PGL_d0():
    return 16*(3**(n/2)*4096**(n/16))

# This function is $P(K_{2,n},m)$.
def PG(n,m):
    return m * ((m-1)**(n)) + m*(m-1) * ((m-2)**(n))

stop = False
chromatic_polynomial = 0
while (n >= 1 and n <= 24):
    chromatic_polynomial = PG(n,m)
    for x in range(n+1):
        if (chromatic_polynomial > PGL_d1(x) 
        or chromatic_polynomial > PGL_d0()):
            stop = True
    if (not stop):
        print("n = " + str(n) + " is good")
        n += 1
    else:
        break
\end{verbatim}

Our final program is needed for the proof of Statement~(iii) of Theorem~\ref{thm: casework}.

\begin{verbatim}
import math

n=1
m=5 

# This function is the lower bound for $P(G,L)$ when $d=2$.
def PGL_d2(x, y, z, w, v):
    return (4*(4**(x) * 4**(y) * 20**(z/2) * 20**(w/2) * 5**(v)) + 
    12*(12**(x/2) * 12**(y/2) * 2880**(z/6) * 2880**(w/6) * 5120**(v/6)) + 
    9*(4**(x) * 230400**(y/9) * 48**(z/3) * 103680**(w/9) * 36**(v/3)))

# This function is the lower bound for $P(G,L)$ when $d=1$.
def PGL_d1(x, y, z, w, v):
    return ((4**(x) * 4**(y) * 4**(z) * 5**(w) * 5**(v)) + 
    8*(12**(x/2) * 12**(y/2) * 12**(z/2) * 128000**(w/8) * 128000**(v/8)) + 
    16*(4**(x) * 3538944000**(y/16) * 240**(z/4) * 192**(w/4) * 34560**(v/8)))

# This function is the lower bound for $P(G,L)$ when $d=0$.
def PGL_d0(x, y, z):
    return 25 * (4**(x) * (3**(4/25*y)*4**(17/25*y)*5**(4/25*y)) * 
		(3**(6/25*z)*4**(13/25*z)*5**(6/25*z)))

# This function is $P(K_{2,n},m)$.
def PG(n,m):
    return m * ((m-1)**(n)) + m*(m-1) * ((m-2)**(n))

stop = False
chromatic_polynomial = 0
while (n >= 1 and n <= 43):
    chromatic_polynomial = PG(n,m)
    for x in range(n+1):
        for y in range(n+1-x):
            for z in range(n+1-x-y):
                for w in range(n+1-x-y-z):
                    v = n-x-y-z-w
                    if (chromatic_polynomial > PGL_d2(x,y,z,w,v) 
                    or chromatic_polynomial > PGL_d1(x,y,z,w,v)):
                        stop = True
    for x in range(n+1):
        for y in range(n+1-x):
            z = n-x-y
            if (chromatic_polynomial > PGL_d0(x,y,z)):
                stop = True
    if (not stop):
        print("n = " + str(n) + " is good")
        n += 1
    else:
        break
\end{verbatim}
\end{document}